\documentclass{article}
\usepackage{makeidx}
\makeindex

\usepackage{url}
\usepackage[dvipsnames]{xcolor}
\usepackage{makeidx}
\usepackage{amsmath}
\usepackage{amssymb}
\usepackage{gastex}
\usepackage{theorem}
\usepackage{ifpdf}
\ifpdf
\usepackage{pgf}
\usepackage{tikz}
\usepackage{todonotes}
\fi
\usepackage{theorem}
\newtheorem{theorem}{Theorem}[section]
\newtheorem{lemma}[theorem]{Lemma}
\newtheorem{proposition}[theorem]{Proposition}
\newtheorem{corollary}[theorem]{Corollary}
{\theorembodyfont{\rmfamily}%
  \newtheorem{example}[theorem]{Example}
   }
\newenvironment{proof}{\noindent\textit{Proof.}}
{\QED\vskip\theorempostskipamount} 
\newenvironment{proofof}[1]{\noindent\textit{Proof
    \protect{#1}.}}
                       {\QED\vskip\theorempostskipamount}
\def\petitcarre{\vrule height4pt width 4pt depth0pt}
\def\QED{\relax\ifmmode\eqno{\hbox{\petitcarre}}\else{%
  \unskip\nobreak\hfil\penalty50\hskip2em\hbox{}\nobreak\hfil
  \petitcarre
  \parfillskip=0pt \finalhyphendemerits=0\par\smallskip}
  \fi}

\DeclareMathOperator{\Card}{Card}
\DeclareMathOperator{\ord}{ord}
\DeclareMathOperator{\End}{End}
\DeclareMathOperator{\Pal}{Pal}
\DeclareMathOperator{\rank}{rank}
\DeclareMathOperator{\Ker}{Ker}
\newcommand\A{\mathcal{A}}
\newcommand\B{\mathcal{B}}
\newcommand\T{\mathcal{T}}

\newcommand\RR{\mathcal{R}}

\newcommand\JJ{\mathcal{J}}
\newcommand\LL{\mathcal{L}}
\newcommand\HH{\mathcal{H}}
\newcommand\DD{\mathcal{D}}
\newcommand{\edge}[1]{\stackrel{#1}{\rightarrow}}
\newcommand\N{\mathbb {N}}
\newcommand\Z{\mathbb {Z}}
\newcommand{\F}{FG}
   \def\soft#1{\leavevmode\setbox0=\hbox{h}\dimen7=\ht0\advance
    \dimen7 by-1ex\relax\if t#1\relax\rlap{\raise.6\dimen7
    \hbox{\kern.3ex\char'47}}#1\relax\else\if T#1\relax
    \rlap{\raise.5\dimen7\hbox{\kern1.3ex\char'47}}#1\relax
    \else\if d#1\relax\rlap{\raise.5\dimen7\hbox{\kern.9ex
    \char'47}}#1\relax\else\if D#1\relax\rlap{\raise.5\dimen7
    \hbox{\kern1.4ex\char'47}}#1\relax\else\if l#1\relax
    \rlap{\raise.5\dimen7\hbox{\kern.4ex\char'47}}#1\relax
    \else\if L#1\relax\rlap{\raise.5\dimen7\hbox{\kern.7ex
    \char'47}}#1\relax\else\message{accent \string\soft
    \space #1 not defined!}#1\relax\fi\fi\fi\fi\fi\fi} 

\title{Profinite semigroups}
\author{Revekka Kyriakoglou, Dominique Perrin\\
Universit\'e Paris-Est, LIGM}

\begin{document}
\maketitle
\begin{abstract}
We present a survey of results on profinite semigroups and their
link with symbolic dynamics. We develop a series of results, mostly
due to Almeida and Costa and we also include some original results
on the Sch\"utzenberger groups associated to a uniformly recurrent set.
\end{abstract}

\tableofcontents

\section{Introduction}
\label{sec:intro}
The theory of profinite groups originates in the theory
of infinite Galois groups and of $p$-adic analysis
(see~\cite{RibesZalesskii2010}).
The corresponding theory for semigroups
 was considerably developed by Jorge Almeida (see~\cite{Almeida2005}
for an introduction). The initial motivation has been
the theory of varieties of languages and semigroups, put in correspondance
by Eilenberg's theorem. Later, Almeida has initiated the study of
the connexion
between free profinite semigroups and 
symbolic dynamics (see~\cite{Almeida2002}). He has shown 
in~\cite{Almeida2005b} that
minimal subshifts correspond to maximal $\JJ$-classes of the free
profinite monoid. Moreover, the Sch\"utzenberger group of such a $\JJ$-class
is a dynamical invariant of the subshift~\cite{AlmeidaCosta2013}.
Finally, it is shown in~\cite{AlmeidaCosta2015} that if a minimal
system satisfies the tree condition (as defined in~\cite{BertheDeFeliceDolceLeroyPerrinReutenauerRindone2013a}), the 
corresponding group is a free profinite group.

In these notes, we give a gentle introduction to the notions
used in profinite algebra and develop the link with 
minimal sets.

Our motivation to explore the profinite world is the following.
We are interested in the  situation where
we fix a  uniformly recurrent set $F$ on an alphabet $A$
(in general a non rational set, like the set of factors of the Fibonacci
word). We want to study  sets of the form
$F\cap L$ where $L$ is a rational set on $A$, that is the inverse
image through a morphism
$\varphi:A^*\rightarrow M$ of a subset of a finite monoid $M$.
 The aim is thus to develop
a theory of automata observing through the filter of
a non rational set. 

We are particularly interested in sets
$L$ of the form
\begin{enumerate}
\item $L=wA^*w$ for some word $w\in F$ (linked to the complete return words to $w$)
\item $L=h^{-1}(H)$ where $h:A^*\rightarrow G$ is a morphism
onto a finite group $G$ and $H$ is a subgroup of $G$.
\end{enumerate}
This question has been studied successfully in a number of 
cases starting with a Sturmian set $F$ in~\cite{BerstelDeFelicePerrinReutenauerRindone2012} and progessively generalizing to sets called
tree sets in~\cite{BertheDeFeliceDolceLeroyPerrinReutenauerRindone2013a}.

The framework of profinite semigroups allows one to work simultaneously
with all rational sets $L$. This is handled, as we shall see, both
through the definition of an inverse limit and through a topology
on words. This topology is defined by introducing a distance on words:
two words are close if one needs a morphism $\varphi$ on a large monoid $M$
to distiguish them. For example, for any word $x$, the powers $x^{n!}$ of $x$
will not be distinguished by a monoid with less than $n$ elements.
Thus one can consider a limit, called a pseudoword, and denoted
$x^\omega$ which has the same image by all morphisms $\varphi$ onto a finite
monoid $M$.

The moment of inspiration was the derivation
by Ameida and Costa \cite{AlmeidaCosta2015} of a result
on profinite groups (the Sch\"utzenberger group of a tree set is free)
which uses the main result of~\cite{BertheDeFeliceDolceLeroyPerrinReutenauerRindone2013a}. This result gives a global version
of properties like the Finite Index Basis Property, as defined
in~\cite{BertheDeFeliceDolceLeroyPerrinReutenauerRindone2015}.
One of the goals of this paper is to develop this connexion.

We begin with two motivating examples: the first one
concerns $p$-adic and profinite numbers (Section~\ref{sectionpAdic})
and the second one 
the Fibonacci morphism (Section~\ref{sectionFibonacci}).

We give in Section~\ref{sectionProfinite}
an introduction to the basic notions concerning
profinite semigroups. We have chosen a simplified presentation
which uses the class of all semigroups instead of working
inside a pseudovariety of semigroups. It simplifies the statements
but the proofs work essentially in the same way.

In Section~\ref{sectionProfiniteCodes} we describe results concerning
codes in profinite semigroups. The main result, 
from~\cite{MargolisSapirWeil1998}, is that any
finite code
is a profinite code (Theorem~\ref{theoremCodeProfinite}).

In Section~\ref{factorialSets} we introduce uniformly recurrent pseudowords. We
prove a result of Almeida (Theorem~\ref{theoremUniformRecurrence}) showing that uniformly
recurrent pseudowords can be characterized in algebraic terms,
as the $\JJ$-maximal elements of the free profinite monoid.

In Section~\ref{sectionSturmian}, we recall some basic properties of Sturmian
sets and their generalization, the tree sets, introduced in~\cite{BertheDeFeliceDolceLeroyPerrinReutenauerRindone2013a}.

In Section~\ref{sectionReturn},
we prove several results due
to Costa and Almeida concening the presentation of the Sch\"utzenberger group of
a uniformly recurrent set.

In Section~\ref{sectionSchutzenbergerGroups}, we prove a new result
(Thorem~\ref{newTheorem}) concerning the Sch\"utzenberger groups 
of uniformly recurrent sets.
\paragraph{Acknowledgements}
We would like to thank Jorge Almeida and Alfredo Costa for their
help in the preparation of this manuscript. It was written in connexion
with a workshop held in Marne la Vall\'ee on january 20-22 2016 and gathering
around them 
Marie-Pierre B\'eal, Val\'erie Berth\'e, Francesco Dolce, Pavel Heller,
Julien Leroy, Jean-Eric Pin and the autors.
%%%%%%%%%%%%%%%%%%%%%

\section{$p$-adic numbers}\label{sectionpAdic}
We begin with a motivating example (see~\cite{Koblitz1984}
for an introduction to this subject). Let $p$ be a prime
number and let $\Z_p$ denote the ring of $p$-adic integers,
namely the completion of $\Z$ under the $p$-adic metric.
This metric is defined by the norm
\begin{displaymath}
|x|_p=\begin{cases}p^{-\ord_p(x)}&\text{if $x\ne 0$}\\0&\text{otherwise}
\end{cases}
\end{displaymath}
where $\ord(x)_p(x)$ is the largest $n$ such that $p^n$ divides $x$.
For the topology defined by this metric, $\Z_p$ is compact.

Any element  $\gamma\in\Z_p$  has a unique $p$-adic expansion
\begin{displaymath}
\gamma=c_0+c_1p+c_2p^2+\ldots=(\ldots c_3c_2c_1c_0)_p
\end{displaymath}
Note that infinite expansions may represent ordinary integers.
For example
\begin{displaymath}
(\ldots 111)_2=-1. 
\end{displaymath}
We can express the expansion of the elements in $\Z_p$
as
\begin{eqnarray*}
\Z_p&=&\lim\Z/p^n\Z\\
&=&\{(a_n)_{n\ge 1}\in\prod_{n\ge 1}\Z/p^n\Z\mid \text{for all } n,a_{n+1}\equiv a_n
\bmod p^n\}
\end{eqnarray*}
This expresses the ring $\Z_p$ as a projective limit of the rings $\Z/p^n\Z$.
The direct product of all rings $\Z_p$ 
\begin{displaymath}
\hat{\Z}=\prod_p\Z_p
\end{displaymath}
over all prime numbers $p$ is the ring of \emph{profinite integers}\index{profinite!integer}.
Its topolology is that induced by the product. As a product of compact spaces,
$\hat{\Z}$ is itself compact. It does not depend on a particular number $p$
and shares with all rings $\Z_p$ the property of being a compact topological
space. Thus it is a compact covering of all rings $\Z_p$.

One may define it equivalently as the projective limit of all cyclic groups
\begin{eqnarray*}
\hat{\Z}&=&\prod_{n\ge 1}\Z/n\Z\\
&=&\{(a_n)_{n\ge 1}\in\prod_{n\ge 1}\Z/n\Z\mid \text{for all } n|m, a_m\equiv a_n
\bmod n\}
\end{eqnarray*}
or as the projective limit of the cyclic groups $\Z/n!\Z$
\begin{eqnarray*}
\hat{\Z}&=&\prod_{n\ge 1}\Z/n!\Z\\
&=&\{(a_n)_{n\ge 1}\in\prod_{n\ge 1}\Z/n!\Z\mid \text{for all } n, a_{n+1}\equiv a_n
\bmod n!\}
\end{eqnarray*}
The last representation corresponds to an expansion of the form
\begin{displaymath}
\gamma=c_1+c_22!+c_33!+\ldots=(\ldots c_3c_2c_1)_!
\end{displaymath}
with digits $0\le c_i\le i$. This expansion forms the \emph{factorial number
system}\index{factorial number system} (see~\cite{Knuth1998}). 

Note that this time
\begin{displaymath}
-1=(\ldots 321)_!
\end{displaymath}
which holds because $1+2.2!+\ldots+n.n!=(n+1)!-1$, as one may verify by induction on $n$. The profinite topology on $\hat{\Z}$ can also be defined directly
by the norm $|x|_!=2^{-r(x)}$ where $r(x)$ is the largest $n$ such that
$n!$ divides $x$. A sequence converges with respect to this topology if the expansions converge in the usual sense, that is the number of equal digits,
starting from
the right, tends to infinity.

It is possible to define profinite Fibonacci numbers
(see~\cite{Lenstra2005}). Indeed, 
 Fibonacci numbers are
defined by $F_0=0$, $F_1=1$ and $F_n=F_{n-1}+F_{n-2}$. The definition can be extended to negative $n$ by $F_n=(-1)^{n-1}F_{-n}$. Then one may extend the
function $n\mapsto F_n$  to a continuous function from $\hat{\Z}$ into
itself. 
We note that since $n!$ tends to $0$ in $\hat{\Z}$, the
sequence $F_{n!}$ tends to $F_0=0$. Similarly $F_ {n!+2}$ 
and $F_{n!+1}$ tend to $1$. We come back to this in the next section.
%%%%%%%%%%%%%%%%%%%%%%
\section{The Fibonacci morphism}\label{sectionFibonacci}
In the last example, we have seen how a linear recurrent sequence
(the Fibonacci numbers) can interestingly be extended to a topological
limit. As well known, Fibonacci numbers are the lengths
of the Fibonacci words defined inductively by $w_0=b$,
$w_1=a$ and satisfying the recurrence relation
$w_{n+1}=w_n+w_{n-1}$ for $n\ge 1$. Then $|w_n|=F_{n+1}$.
In the same way as one may embed the ring of ordinary integers into the compact
ring of profinite integers, we will see that one may extend the Fibonacci
sequence to a converging sequence of pseudowords whose lengths are profinite
integers.

The \emph{Fibonacci morphism}\index{Fibonacci!morphism} is the morphism $\varphi:A^*\rightarrow A^*$
with $A=\{a,b\}$ defined by $\varphi(a)=ab$ and $\varphi(b)=a$.
The sequence of $\varphi^n(a)$
\begin{eqnarray*}
\varphi(a)&=&ab\\
\varphi^2(a)&=&aba\\
\varphi^3(a)&=&abaab\\
\varphi^4(a)&=&abaababa\\
&\cdots&
\end{eqnarray*}
is  the Fibonacci sequence of words of length equal to the Fibonacci numbers.
 One has $F_n=|\varphi^{n-2}(a)|$ for $n\ge 2$ (and for $n\in \Z$
with appropriate extensions).

The Fibonacci sequence of words converges in the space $A^\N$ to the
Fibonacci infinite word
\begin{displaymath}
x=abaababa\cdots
\end{displaymath}
It is a fixed-point of $\varphi$ in the sense that $\varphi(x)=x$.
From another point of view, this sequence is not convergent. 
Indeed, the terms of the sequence end alternately with $a$ or $b$
and thus can be distinguished by a morphism from $A^*$ into
a monoid with $3$ elements. A sequence converging in this
stronger sense is $\varphi^{n!}(a)$
\begin{eqnarray*}
\varphi(a)&=&ab\\
\varphi^2(a)&=&aba\\
\varphi^6(a)&=&abaababaabaababaababa\\
\varphi^{24}(a)&=&abaababa\cdots abaababa\\
&\cdots&
\end{eqnarray*}
The limit in the sense of profinite topology, to be defined below,
is a pseudoword denoted $\varphi^\omega(a)$ which begins
by the Fibonacci infinite word $(\varphi^n(a))_{n\ge 0}$ and ends with the left infinite word 
$(\varphi^{2n}(a))_{n\ge 0}$.

Its length  in $\hat{\Z}$ is the limit
of the Fibonacci numbers $F_{n!+2}$ which is $F_2=1$. 
We shall come back to this interesting property in Example~\ref{exampleFibonacciPseudo}.

%%%%%%%%%%%%%%%%%%%%%%%%%%%%%%%%%%%
\section{Topological spaces, groups and semigroups}
In this section, we  recall the basic definitions of topology
and the notion of topological semigroup.

\subsection{Topological spaces}
We begin with an introduction to
the basic notions of topology (see~\cite{Willard2004} for example).
A \emph{topological space}\index{topological!space}
\index{space!topological}
 is a set $S$ with a family $\cal F$ of subsets 
such that
\begin{enumerate}
\item[(i)] it contains $\emptyset$ and $S$,
\item[(ii)] it is closed under union,
\item[(iii)] it is closed under finite intersection
\end{enumerate}
The elements of $\cal F$ are called 
\emph{open} sets\index{open set}\index{set!open}.
The complement of an open set is called a 
\emph{closed} set\index{closed set}\index{set!closed}.
 A \emph{clopen}
set\index{clopen!set}\index{set!clopen} is both open and closed.

For any set $S$, the \emph{discrete topology}\index{discrete topology}
\index{topology!discrete} is the topology for which all subsets are open.

For any subset $X$ of a topological set $S$, the topology \emph{induced}
\index{induced topology}\index{topology!induced}
on $X$ by the topology of $S$ corresponds to the family of open sets
$X\cap Y$ for $Y$ open in $S$.

A \emph{limit point}\index{limit point} of a subset $X$ of a topological space $S$
is a point $x$ such that any open set containing $x$ contains a
point of $X$ distinct of $x$. As a particlular case,
it is said to be an \emph{accumulation}
point \index{accumulation point}
if every open set containing $x$ contains an infinite
number of points of $X$.

The \emph{closure}\index{closure} of a subset $X$ of $S$, denoted $\Bar{X}$
 is the union of $X$ and
all its limit points. It is also the smallest closed set
containing $X$. A set is closed if and only if $X=\bar{X}$.
The set $X$ is \emph{dense}\index{dense} in $S$ is $\bar{X}=S$.

The notion of limit point can also be defined for a sequence. A point
$x$ is said to be a \emph{limit point}\index{limit point} of the sequence $(x_n)_{n\ge 0}$
if any open set containing $x$ contains all but a finite number
of terms of the sequence. It is said to be an \emph{accumulation point}\index{accumulation point}
of the sequence if every open set containing $x$ contains
an infinite number of elements of the sequence.
A limit point of the sequence $(x_n)_{n\ge 0}$
is a limit point of the set $\{x_n\mid n\ge 0\}$.

%A sequence $(x_n)_{n\ge 0}$ of elements of a topological space $S$
%\emph{converges} to a limit $x$ if every open set containing
%$x$ contains all but a finite number of the $x_n$.
%A limit point of the set formed by the $x_n$ is a limit
%of a subsequence of the sequence $x_n$.

A map $\varphi:X\rightarrow Y$ between topological spaces $X,Y$ is
\emph{continuous}\index{continuous map} if for any open set $U\subset Y$, the set
$\varphi^{-1}(U)$ is open in $X$.

A \emph{basis}\index{open set!basis of} of the family of open sets is a family $\cal B$ of sets
such that any open set is a union of elements of $\cal B$.

Given a family of topological spaces $X_i$ indexed by a set $I$,
the \emph{product topology}\index{product topology}\index{topology!product} on the
direct product $X=\prod_{i\in I}X_i$ is defined as the coarsest
topology such that the projections $\pi_i:X\rightarrow X_i$
are continuous. A basis of the family of open sets is the 
family of \emph{open boxes}\index{open box}\index{box!open}, that is  sets of the form $\prod_{i\in I}U_i$
where the $U_i$ are open sets such that
$U_i\ne X_i$ only for a finite number of indices $i$.

\begin{example}
The set $A^\N$ of infinite words over $A$ is a topological space
for the discrete topology on $A$. It is not a semigroup
(a big motivation for introducing pseudowords!). However, there
is a well defined product of finite words and infinite words.
The open sets are
the sets of the form $XA^\N$ for $X\subset A^*$.
\end{example}

Metric spaces form a vast family of topological spaces.
A \emph{metric space}\index{metric space}\index{space!metric} is a space $S$ with a function $d:S\times S\rightarrow
\mathbb R$, called a \emph{distance}\index{distance}, such that for all $x,y,z\in S$,
\begin{enumerate}
\item[(i)] $d(x,y)=0$ if and only if $x=y$,
\item[(ii)] $d(x,y)=d(y,x)$
\item[(iii)] $d(x,z)\le d(x,y)+d(y,z)$ (triangle inequality).
\end{enumerate}
Any metric space can be considered as a topological space, considering
as open sets the unions of \emph{open balls}\index{open ball}
 $B_\varepsilon(x)=\{y\in S\mid d(x,y)<\varepsilon\}$ for $x\in S$ and $\varepsilon\ge 0$.

For example, the set ${\mathbb R}^n$ is a metric space for the Euclidean distance.

A topological space is a \emph{Hausdorff space}\index{Hausdorff space}\index{space!Hausdorff} if
any two distinct points belong to disjoint open sets.
A metric space is a Hausdorff space. In a Hausdorff space, every limit
point of a set is an accumulation point.

A topological space is \emph{compact}\index{compact space}\index{space!compact} if it is a Hausdorff space and if
from any family of open sets whose union is $S$, one may extract
a finite subfamily with the same property.

A closed subset of a compact space is compact and, by Tychononoff's theorem,
any product of compact spaces
is compact.

One may verify that in a compact space every infinite set
has an accumulation point (the converse is true in metric spaces).
Indeed, if $X$ is an infinite subset of a compact space $S$
without accumulation point,
there is for every $x\in S$ an open set containing $x$ which contains only
a finite number of elements of $X$. For a finite set $F\subset X$,
denote by $O_F$ the union of the $O_x$ such that $O_x\cap X=F$.
Then the sets $O_F$ form a family of open sets whose union is $S$
every finite subfamily of which intersects at most finitely
many points in $X$. Thus no finite subfamily may cover $S$, a contradiction.

The clopen sets in a product of compact spaces $S_i$ are the
finite unions of \emph{clopen boxes}\index{clopen!box}\index{box!clopen},
that is the sets of the form $\Pi_{i\in I}K_i$ where each $K_i$
is clopen in $S_i$ and $K_i=S_i$ for all but a finite number
of indices $i$.

\subsection{Topological semigroups}
A \emph{semigroup}\index{semigroup} is a set with an associative operation.
A \emph{monoid}\index{monoid} is a semigroup with a neutral element.

A \emph{topological semigroup}\index{topological!semigroup}\index{semigroup!topological} is a semigroup $S$ endowed with a topology
such that the semigroup operation $S\times S\rightarrow S$ is
continuous. A topological monoid is a topological semigroup with
identity.
\begin{example}
A finite semigroup can always be viewed as a topological semigroup
under the discrete topology.
\end{example}
\begin{example}
As a less trivial example, the set $\mathbb R$ of nonnegative real numbers is a topological semigroup for the addition and the interval $[0,1]$ is a topological
semigroup for the multiplication.
\end{example}
A compact monoid is a topological monoid which is compact (as a topogical
space). Note that we assume a compact space to satisfy Hausdorff separation axiom
(any two distinct points belong to disjoint open sets).
Note also the following elementary property of compact monoids.
Recall that, is a monoid $M$, $u\in M$ is a \emph{factor}\index{factor}
of $v\in M$ if $v\in MuM$.

%\begin{proposition}\label{propositionFactors}
%The set of factors of an element of a compact monoid is closed.
%\end{proposition}
%\begin{proof}
%Let $M$ be a compact monoid and let $(u_n)_{n\ge 0}$ be a sequence of factors
%of $x\in M$ converging to some $u\in M$. Let $p_n,q_n$ be such that $x=p_nu_nq_n$ for all $n\ge 1$.
%Since $M$ is compact, the sequences $(p_n),(q_n)$ have converging
%subsequences. If $p,q$ are the limits of these subsequences, we have
%$x=puq$ and thus $u$ is a factor of $x$.
%\end{proof}

\subsection{Topological groups}
A \emph{topological group}\index{topological!group}\index{group!topological}
 is a group with a topology such that the multiplication
and taking the inverse are continuous operations. It is in particular
a topological semigroup.

\begin{example}
The set $\mathbb R$ of real numbers
with the usual topology is a topological group under addition.
\end{example}

Any closed subgroup is a topological group for the induced topology.
Moreover, since multiplication is continuous, the cosets 
$Hg$ of an open (resp. closed) subgroup are open (resp. closed).

Every open subgroup $H$ of a topological group $G$ is also closed
since its complement is the union of all cosets $Hg$ for $g\in G\setminus H$
which are open.

The following is~\cite[Lemma 2.1.2]{RibesZalesskii2010}.
\begin{proposition}
In a compact group, a subgroup is open if and only if it is
closed and of finite index.
\end{proposition}
\begin{proof}
Assume that $H$ is an open subgroup of $G$. We have already
seen that $H$ is also closed.  The union
of the cosets of $H$ form a covering by open sets. Since $G$
is compact, there is a finite subfamily covering $G$ and thus 
$H$ has finite index. 

Conversely, if $H$ is a closed subgroup of finite index, then the complement
of $H$ is the union of the cosets $Hg$ for $g\notin H$ and thus
$H$ is open.
\end{proof}
%%%%%%%%%%%%%%%%%%%%%%%%%%%%
\section{Profinite semigroups}\label{sectionProfinite}
In this section, we introduce the notions of profinite semigroup
and of profinite group. We begin with the notion of projective limit.
\subsection{Projective limits}
We want to define profinite semigroups as some kind of limit of finite
semigroups in such a way that properties true in all finite semigroups
will remain true in profinite semigroups. For this we need the notion
of projective limit.

A \emph{projective system}\index{projective!system}
(or \emph{inverse system}\index{inverse!system})
of semigroups  is given by 
\begin{enumerate}
\item[(i)] a directed set $I$, that is a poset in which any two elements
have a common upper bound,
\item[(ii)] for each $i\in I$, a topological
semigroup $S_i$,

\item[(iii)] for each pair $i,j\in I$ with $i\ge j$, a \emph{connecting morphism}\index{morphism!connecting} $\psi_{i,j}:S_i\rightarrow S_j$ such that $\psi_{i,i}$ is the identity
on $S_i$ and for $i\ge j\ge k$, $\psi_{i,k}=\psi_{i,j}\circ\psi_{j,k}$.
\end{enumerate}
\begin{example}\label{exampleCyclic}
Let $I$ be the set of natural integers ordered by divisibility: $n\ge m$
if $m|n$.
The family of cyclic groups $(\Z/n\Z)_{n\in I}$ 
forms a projective system for the morphisms $\psi_{n,m}$ defined
by $\psi_{n,m}(x)=x\bmod m$. 

In the same way, the family of cyclic groups $(\Z/n!\Z)_{n\ge 0}$
indexed by the set $I$ of natural integers with the natural order
is a projective system. 
\end{example}

The \emph{projective limit}\index{projective!limit}
(or \emph{inverse limit}\index{inverse!limit}) of this projective system is a topological
semigroup  $S$ 
together with  morphisms $\Phi_i: S\rightarrow S_i$
such that  for all $i,j\in I$ with
$i\ge j$, $\psi_{i,j}\circ \Phi_i=\Phi_j$, and
 for any topological semigroup  $T$
and  morphisms $\Psi_i: T\rightarrow S_i$
such that for all $i,j\in I$ with
$i\ge j$, $\psi_{i,j}\circ \Psi_i=\Psi_j$, there exists a morphism
$\theta:T\rightarrow S$ such that $\Phi_i\circ \theta=\Psi_i$
for all $i\in I$. 
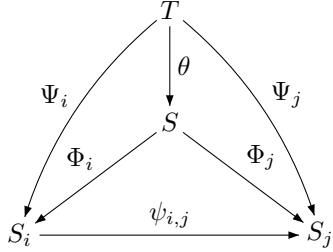
\begin{figure}[hbt]
\centering
\gasset{Nframe=n,Nadjust=wh}
\begin{picture}(40,45)
\node(phi_i)(0,0){$S_i$}
\node(Phi)(20,15){$S$}
\node(Psi)(20,30){$T$}
\node(phi_j)(40,0){$S_j$}

\drawedge(phi_i,phi_j){$\psi_{i,j}$}
\drawedge[ELside=r](Phi,phi_i){$\Phi_i$}
\drawedge(Phi,phi_j){$\Phi_j$}
\drawedge(Psi,Phi){$\theta$}
\drawedge[curvedepth=-3,ELside=r](Psi,phi_i){$\Psi_i$}
\drawedge[curvedepth=3](Psi,phi_j){$\Psi_j$}
\end{picture}
\caption{The projective limit.}
\end{figure}

The uniqueness of the projective limit can be verified (``as a standard diagram
chasing exercise''~\cite{Almeida2005}). The existence can be proved by considering the subsemigroup $S$ of the product $\prod_{i\in I}S_i$ consisting of all
$(s_i)_{i\in I}$ such that, for all $i,j\in I$ with $i\ge j$,
\begin{displaymath}
\psi_{i,j}(s_i)=s_j
\end{displaymath}
endowed with the product topology. The maps $\Phi_i:S\rightarrow S_i$
 are the projections, that is, if $s=(s_i)_{i\in I}$, then $\Phi_i(s)=s_i$. 

One defines in the same way a projective system of monoids or groups and
a projective limit of monoids or groups. For a projective system
of monoids, one has to take all morphisms as monoid morphisms
and similarly for groups (actually a monoid morphism between
groups is already a group morphism).
\begin{example}
The projective limit of the family of cyclic groups (Example~\ref{exampleCyclic})
is the group of profinite integers.
\end{example}
A variant of this construction allows to specify a fixed generating
set for all semigroups.
An $A$-\emph{generated topological semigroup}
\index{topological!$A$-generated semigroup} is a 
topological semigroup $S$ together with a
mapping $\varphi:A\rightarrow S$
 whose image generates a subsemigroup dense
in $S$. A morphism between $A$-generated topological semigroups
$\varphi:A\rightarrow S$ and $\psi:A\rightarrow T$ is a 
continuous morphism $\theta:S\rightarrow T$ such that $\theta\circ\varphi=\psi$.
We denote $\theta:\varphi\rightarrow \psi$ such a morphism.
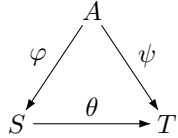
\begin{figure}[hbt]
\centering
\gasset{Nframe=n,Nadjust=wh}
\begin{picture}(20,15)
\node(S)(0,0){$S$}\node(T)(20,0){$T$}
\node(A)(10,15){$A$}

\drawedge[ELside=r](A,S){$\varphi$}
\drawedge(A,T){$\psi$}
\drawedge(S,T){$\theta$}
\end{picture}
\caption{A morphism of $A$-generated semigroups}
\end{figure}

A projective system
in this category of objects is given by 
 a directed set $I$ and
 for each $i\in I$, an $A$-generated topological
semigroup $\varphi_i:A\rightarrow S_i$,
and for each pair $i,j\in I$ with $i\ge j$, a connecting morphism
 $\psi_{i,j}:\varphi_i\rightarrow \varphi_j$ such that $\psi_{i,i}$ is the identity
on $S_i$ and for $i\ge j\ge k$, $\psi_{i,k}=\psi_{i,j}\circ\psi_{j,k}$.

The projective limit
of this projective system is an
$A$-generated topological
semigroup  $\Phi:A\rightarrow S$ 
together with  morphisms $\Phi_i: \Phi\rightarrow \varphi_i$
such that  for all $i,j\in I$ with
$i\ge j$, $\psi_{i,j}\circ \Phi_i=\Phi_j$, and
 for any $A$-generated topological semigroup  $\Psi:A\rightarrow T$
and  morphisms $\Psi_i: \Psi\rightarrow \varphi_i$
such that for all $i,j\in I$ with
$i\ge j$, $\psi_{i,j}\circ \Psi_i=\Psi_j$,  there is a morphism
$\theta:\Psi\rightarrow \Phi$ such that $\Phi_i\circ \theta=\Psi_i$
for all $i\in I$.

\begin{figure}[hbt]
\centering
\gasset{Nframe=n,Nadjust=wh}
\begin{picture}(80,45)
\put(0,0){
\node(phi_i)(0,0){$\varphi_i$}
\node(Phi)(20,15){$\Phi$}
\node(Psi)(20,30){$\Psi$}
\node(phi_j)(40,0){$\varphi_j$}

\drawedge(phi_i,phi_j){$\psi_{i,j}$}
\drawedge[ELside=r](Phi,phi_i){$\Phi_i$}
\drawedge(Phi,phi_j){$\Phi_j$}
\drawedge(Psi,Phi){$\theta$}
\drawedge[curvedepth=-3,ELside=r](Psi,phi_i){$\Psi_i$}
\drawedge[curvedepth=3](Psi,phi_j){$\Psi_j$}
}
\put(50,0){
\node(S_i)(0,0){$S_i$}
\node(S)(20,15){$S$}
\node(T)(20,30){$T$}
\node(S_j)(40,0){$S_j$}
\node(A)(20,45){$A$}

\drawedge(S_i,S_j){$\psi_{i,j}$}
\drawedge[ELside=r](S,S_i){$\Phi_i$}
\drawedge(S,S_j){$\Phi_j$}
\drawedge(T,S){$\theta$}
\drawedge[curvedepth=-3,ELside=r](T,S_i){$\Psi_i$}
\drawedge[curvedepth=3](T,S_j){$\Psi_j$}
\drawedge(A,T){$\Psi$}
\drawedge[curvedepth=-5,ELside=r](A,S_i){$\varphi_i$}
\drawedge[curvedepth=5](A,S_j){$\varphi_j$}
}
\end{picture}

\caption{The projective limit of a family of $A$-generated semigroups.}
\end{figure}
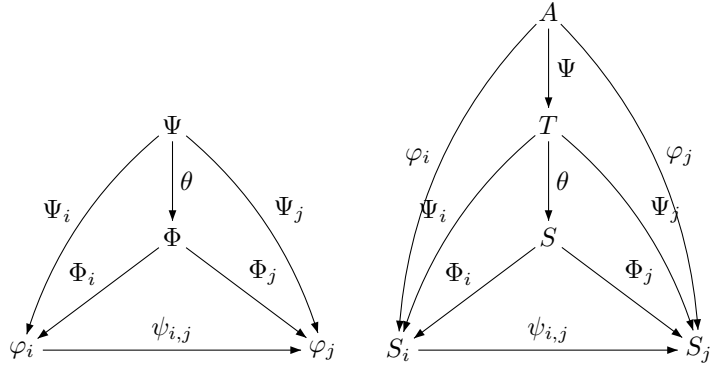

%%%%%%%%%%%%%%%%%%%%%%%%%%%%%%%%%%%
\subsection{Profinite semigroups}
A \emph{profinite semigroup}\index{profinite!semigroup}\index{semigroup!profinite} is a projective limit of a projective system
of finite semigroups. 
\begin{example}
The ring of profinite integers is a profinite group.
\end{example}
A profinite semigroup is compact. Indeed, let $S$ be the
projective limit of a family $S_i$ of finite semigroups.
Then $S$ is a closed submonoid of a direct product of the finite
and thus
compact semigroups $S_i$ and thus is compact.

A topological space is \emph{connected}\index{connected space}\index{space!connected} if it is not the union of two
disjoint open sets. A subset of a topological space is connected if
it is connected as a subspace, that is it cannot be covered by the
union of two disjoint open sets. Every topological space decomposes as
a union of disjoint connected subsets, called its \emph{conneted components}.

A topological space is
\begin{enumerate}
\item[(i)] 
\emph{totally disconnected}\index{totally disconnected}\index{space!totally disconnected}
if its connected components are singletons, 
\item[(iii)] \emph{zero-dimensional}\index{zero-dimensional}\index{space!zero-dimensional} if it admits a basis consisting of clopen
sets. 
\end{enumerate}
The term zero-dimensional is by reference to a notion of dimension in topological spaces (the Lebesgue dimension).
The following result, from~\cite{Almeida2005} gives
a possible direct definition of profinite semigroup without
using projective limits. A topological semigroup is
\emph{residually finite}\index{residually finite!semigroup}\index{semigroup!residually finite} if for any $u,v\in S$ there
exists a continuous morphism $\varphi:S\rightarrow M$ into
a finite semigroup $M$ such that $\varphi(u)\ne\varphi(v)$.
%A \emph{subdirect product} of a family
%$(S_i)_{i\in I}$ of semigroups is a subsemigroup $S$ of the direct product
%$\prod_{i\in I}S_i$ such that the projections $\pi_i:S\rightarrow S_i$
%are surjective.

\begin{theorem}\label{theoremProfiniteSemigroups}
The following conditions are equivalent for a compact semigroup $S$.
\begin{enumerate}
\item[\rm (i)] $S$ is profinite,
\item[\rm (ii)] $S$ is residually finite as a topological semigroup,
\item[\rm (iii)] $S$ is a closed subsemigroup of a direct
 product of finite semigroups,
\item[\rm (iv)] $S$ is totally disconnected,
\item[\rm (v)] $S$ is zero-dimensional.
\end{enumerate}
\end{theorem}
The explicit construction of the projective limit shows that
(i)$\Rightarrow$ (ii) and (ii)$\Rightarrow$ (iii) results from the definitions.
For (iii)$\Rightarrow$ (i), see~\cite{Almeida2005}.
Since a product of totally disconnected spaces is totally disconnected,
we have (iii)$\Rightarrow$ (iv). The equivalence (iv)$\Leftrightarrow$ (v)
holds for any compact space. Finally, the implication (v)$\Rightarrow$ (ii)
results from Hunter's Lemma (see~\cite{Almeida2005}).

\begin{corollary}\label{corollaryClosedSubsemigroup}
The class of profinite semigroups is closed under taking
 closed subsemigroups and direct product.
\end{corollary}

%The following is known as Hunter's Lemma (see~\cite{Almeida2005}).
%We say that a congruence on a topological semigroup is clopen
%if its classes are clopen. In a compact semigroup, a clopen congruence
%has finite index.
%\begin{lemma}\label{lemmaHunter}
%If $S$ is a compact zero-dimensional semigroup and  $K$
%is a clopen subset of $S$ then the syntactic congruence of $K$
%is clopen, and therefore it has finitely many classes.
%\end{lemma}
%\begin{proof}

%\end{proof}
The following is \cite[Proposition 3.5]{Almeida2005}.
A subset $K$ of a semigroup $S$ is \emph{recognized}\index{set!recognized}
 by a morphism
$\varphi:S\rightarrow M$ if $K=\varphi^{-1}\varphi(K)$.
\begin{proposition}\label{propositionClopen}
Let $S$ be a profinite semigroup. A subset $K\subset S$ is clopen if and only
if it is recognized by a continuous morphism $\varphi:S\rightarrow T$ into a finite
semigroup $T$.
\end{proposition}
\begin{proof}
The condition is sufficient since the set $K$ is the inverse image under a continuous function of a clopen set. Conversely assume that $K$ is clopen.
Since $S$ is profinite, it is by Theorem\ref{theoremProfiniteSemigroups}
  a closed subsemigroup of a direct product $\prod_{i\in I}S_i$
of finite semigroups $S_i$. Since $K$ is clopen, it is a finite
union of clopen boxes, that is a finite union
of sets $\Pi_{i\in I}K_i$ with $K_i=S_i$ for all but a finite number
of indices $i$. Thus there is a finite set
$F\subset I$ such that being in $K$ only depends on the coordinates 
in $F$. Then the projection $\varphi:S\rightarrow \Pi_{i\in F}S_i$ is a continuous
morphism into a finite semigroup which recognizes $K$. Indeed, for $s\in\varphi^{-1}\varphi(K)$, there is $t\in K$ such that  $\varphi(s)=\varphi(t)$.
Since $s$ and $t$ agree on their $F$-coordinates, we have $s\in K$.
\end{proof}
%%%%%%%%%%%%
\subsection{Profinite groups}
A \emph{profinite group}\index{profinite!group}\index{group!profinite} is a projective limit
of a projective system of finite groups.

A topological group which is a profinite semigroup is actually a
profinite group. Indeed, let $S$ be
the projective limit of 
a family $S_i$ of finite semigroups. We may assume
that the morphisms $\Phi_i:S\rightarrow S_i$
are surjective. If $S$ is a group, since the image of a group by a semigroup
morphism is a group, each $S_i$ is a finite group and thus $S$ is
a profinite group.

 Any profinite group is a compact group. Indeed, it is a closed subgroup
of a direct product of finite and thus compact groups.

A simple example of a compact group which is not profinite
is the multiplicative group $[0,1]$, which has no nontrivial
image which is a finite group.

 A closed  subgroup
of a profinite group is profinite.
Indeed, let $H$ be a closed subgroup
of a profinite group $G$.
Then $H$ is a closed subsemigroup of $G$ and thus it is a profinite
semigroup by Corollary~\ref{corollaryClosedSubsemigroup}, whence a profinite group.

A group $G$ is \emph{Hopfian}\index{group!Hopfian}\index{Hopfian group}
if every endomorphism of $G$ which is onto is an isomorphism. The
following is~\cite[Proposition 2.5.2]{RibesZalesskii2010}.
It is an analogue property for  profinite groups.

\begin{proposition}\label{propositionHopfian}
Let $G$ be a finitely generated profinite group. and let
$\varphi:G\rightarrow G$ be a continuous surjective morphism.
Then $\varphi$ is an isomorphism.
\end{proposition}
\begin{proof}
We show that $\Ker(\varphi)$ is contained in any subgroup of finite
index of $G$. Since $G$ is profinite, it will imply that $\Ker(\varphi)=\{1\}$.

For each finite group $F$, since $G$ is finitely generated, there
is a finite number of morphisms from $G$ into $F$. Thus there
is only a finite number of morphisms from $G$ into a finite group
of order $n$ and thus also a finite number of normal subgroups of index $n$
of $G$.

Let ${\mathcal U}_n$ be the family of normal subgroups of $G$ of index $n$.
Let $\Phi:{\mathcal U}_n\rightarrow {\mathcal U}_n$ be defined by
$\Phi(U)=\varphi^{-1}(U)$. Then $\Phi$ is an injective map
from a finite set into itself and thus it is a bijection.

Let $U$ be a normal subgroup of index $n$ of $G$. Then, since $\Phi$ is
surjective, we have $U=\varphi^{-1}(V)$ for som $V\in{\mathcal U}_n$.
This implies that $\Ker(\varphi)\subset U$, which was to be proved. 
\end{proof}
%%%%%%%%%%%%%%%%%%
\subsection{Endomorphisms of profinite semigroups}
For a topological semigroup, we denote by $\End(S)$ the monoid of all its
continuous endomorphisms. We consider $\End(S)$ as a topological monoid
for the pointwise convergence.

The following result is~\cite[Theorem 4.14]{Almeida2005}.
\begin{theorem}\label{theoremEndomorphismMonoid}
Let $S$ be a finitely generated profinite semigroup. Then 
$\End(S)$ is a profinite monoid.
\end{theorem}
%%%%%%%%%%%%%%%%%%%%ù
\subsection{The $\omega$ operator}
Recall that an \emph{idempotent}\index{idempotent} in a semigroup $S$
is an element $e\in S$ such that $e^2=e$.

In a finite semigroup $S$, the semigroup generated by $s\in S$
can be represented as in Figure~\ref{fig0_01} (the frying pan)
with the
\emph{index}\index{index} $i$ and the \emph{period}\index{period} $p$
such that $s^{i+p}=s^i$.
It contains a unique idempotent, which is of the form $s^{np}$
with $n$ such that $np\ge i$. Thus, $\Z/p\Z$ is a maximal subroup
of the semigroup generated
by $s$, which coincides with its minimal ideal.
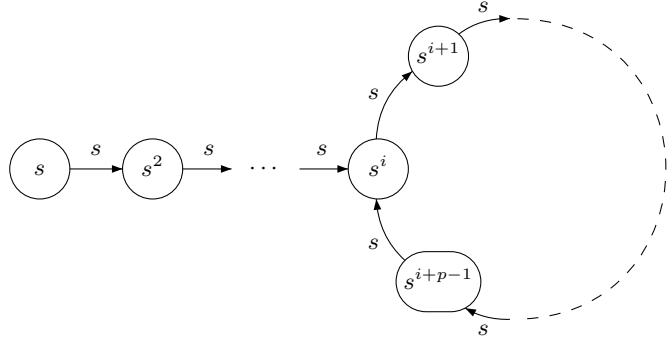
\begin{figure}[hbt]
\centering
\begin{picture}(100,45)(0,-20)
%  \gasset{Nw=6,Nh=6}
\small
%\put(0,-20){\framebox(100,45){}}
  %\node(1)(0,0){$\vphantom(1$}
  \node(2)(15,0){$\vphantom(s$}
  \node(3)(30,0){$\vphantom(s^2$}
  \node[Nw=9,Nframe=n](4)(45,0){$\cdots$}
  \node(5)(60,0){$\vphantom(s^i$}
  \node(6)(68,15){$\vphantom(s^{i+1}$}
  \node[Nframe=n,Nw=1](7)(78,20){}
  \node[Nframe=n,Nw=1](8)(78,-20){}
  \node[Nadjust=w](9)(68,-15){$\vphantom(s^{i+p-1}$}
  %\drawedge(1,2){$\vphantom(s$}
  \drawedge(2,3){$\vphantom(s$}
  \drawedge(3,4){$\vphantom(s$}
  \drawedge(4,5){$\vphantom(s$}
  \drawedge[curvedepth=3.3](5,6){$s$}
  \drawedge[curvedepth=1.6,ELpos=70](6,7){$s$}
  \drawedge[curvedepth=1.6,ELpos=30](8,9){$s$}
  \drawedge[curvedepth=3.3](9,5){$s$}
\drawbpedge[AHnb=0,dash={1.5}0](7,0,27,8,0,27){}
\end{picture}
\caption{The semigroup generated by $s$.}
\label{fig0_01}
\end{figure}

In a compact semigroup $S$, just as in a finite semigroup,
the closure of the semigroup generated by an element
$s\in S$ contains a unique idempotent, denoted $s^\omega$. If $S$ is profinite, it is the limit
of the sequence $s^{n!}$.

We note that in any profinite group, one has $x^\omega=1$ since 
the neutral element is the only idempotent of $G$.

In general, the index can be a finite integer $i$, in which case
we have $s^{\omega+i}=s^\omega$. It can also be
infinite and equal to $\omega$, as in the semigroup $\hat{\N}$.
 Independently, the period can also be finite or infinite.

We also note that for any endomorphism $\varphi$ of a profinite monoid $S$,
the endomorphism $\varphi^\omega$ is a well defined endomorphism of $S$.
%%%%%%%%%%%%%%%%%
\subsection{The free profinite monoid}
Consider the projective system
formed by representatives of isomorphism classes of 
all $A$-generated finite monoids (the finite monoids
are considered as topological monoids for the discrete topology). For $\varphi:A\rightarrow M$
and $\psi:A\rightarrow N$, one has $\varphi\ge \psi$ if
there is a morphism $\mu:M\rightarrow N$ such that $\mu\circ\varphi=\psi$.
Note that $\varphi,\psi,\mu$ have to be surjective.

The \emph{free profinite monoid}\index{free profinite!monoid}\index{monoid!free profinite} on a finite alphabet $A$, denoted $\widehat{A^*}$ is the projective limit of this family. It has the following universal
property (see Figure~\ref{figureUniversal}).

\begin{proposition}
The natural mapping $\iota:A\rightarrow \widehat{A^*}$ is such that
for any map $\varphi:A\rightarrow M$ into a profinite monoid there
exists a unique continuous morphism $\hat{\varphi}:\widehat{A^*}\rightarrow M$
such that $\hat{\varphi}\circ\iota=\varphi$.
\end{proposition}
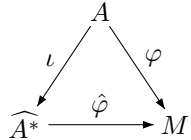
\begin{figure}[hbt]
\centering
\gasset{Nframe=n,Nadjust=wh}
\begin{picture}(20,15)
\node(S)(20,0){$M$}\node(T)(0,0){$\widehat{A^*}$}
\node(A)(10,15){$A$}

\drawedge(A,S){$\varphi$}
\drawedge[ELside=r](A,T){$\iota$}
\drawedge(T,S){$\hat{\varphi}$}
\end{picture}
\caption{The universal property of $\widehat{A^*}$.}
\label{figureUniversal}
\end{figure}
The elements of $\widehat{A^*}$ are called
\emph{pseudowords}\index{pseudoword} and the elements of $\widehat{A^*}\setminus A^*$ are called
\emph{infinite pseudowords}.

The free profinite monoid on one generator is commutative.
Its image  by the
map $a^n\mapsto n$ is the monoid of \emph{profinite natural numbers}
\index{profinite!natural number},
denoted $\hat{\N}$.

The topology induced $A^*$ by the profinite topology on $\widehat{A^*}$
is discrete. Indeed,
if $u\in A^*$ is a word of length $n$, the quotient of $A^*$
by the ideal formed by the words of length greater than $n$ is a finite
monoid. The congruence class of $u$ for the corresponding quotient is 
reduced to $u$.

The \emph{length} $|x|$ of a pseudoword $x\in \widehat{A^*}$
is a profinite natural
number. The map $x\in \widehat{A^*}\rightarrow |x|\in \hat{\N}$ is the continuous
morphism $\lambda$ such that $\lambda(a)=1$ for every $a\in A$.
\begin{example}
The length of $xy^\omega$ is $|x|+\omega$.
\end{example}
\subsection{The free profinite group}
Likewise, the \emph{free profinite group}\index{free profinite!group}
\index{group!free profinite}, denoted $\widehat{\F(A)}$
 is the projective limit of the projective system formed by the
isomorphism classes of $A$-generated finite groups.

The topology on the free group $\F(A)$ induced by the topology
of $\widehat{\F(A)}$ is not discrete. Indeed, for any
$x$ in $\F(A)$, the sequence $x^{n!}$ tends to $1$. Thus $A^*$
is dense in $\widehat{\F(A)}$ and there is an onto homomorphism
from $\widehat{A^*}$ onto $\widehat{\F(A)}$.

The topology induced on $\F(A)$ by the topology of  $\widehat{\F(A)}$
is also called the \emph{Hall topology}\index{Hall topology}. It has been indeed
introduced by M. Hall in~\cite{Hall1950}. Note that, since $A^*$
is embedded in $\F(A)$, we actually
have two topologies on $A^*$ respectively induced by the topologies
of $\widehat{A^*}$ and $\widehat{\F(A)}$. To distinguish them,
the first one is called the \emph{pro-$M$ topology}\index{topology!pro-$M$} and the
second one the \emph{pro-$G$ topology}\index{topology!pro-$G$}. The first one is strictly
stronger than the second one. The pro-$G$ topology on $A^*$
was introduced by Reutenauer in~\cite{Reutenauer1979}.

The image of the free profinite group on one generator $a$
by the map $a^n\mapsto n$ is the group $\hat{\Z}$ of profinite
integers (see Section~\ref{sectionpAdic}). The \emph{length} $|x|$
of an element $x$ of $\widehat{\F(A)}$ is a profinite integer.
The map $x\in \widehat{\F(A)}\rightarrow |x|\in \hat{\Z}$
is the unique continuous morphism such that $|a|=1$
for every $a\in A$. In particular $|a^{-1}|=-1$.

The following result is from~\cite[p. 131]{Hall1950}. The property is not
true for non finitely generated subgroups. The classical example
is the commutator subgroup $F'$ of a finitely generated free group
$F$, which is known to be a free group of countable infinite rank.
 In the topology induced on $F'$ by the profinite topology on $F$,
there are only countably many open subgroups while the number of
open subgroups in the profinite topology of $F'$ is uncountable.
\begin{proposition}\label{propositionFGClosed}
Any finitely generated subgroup $H$ of $\F(A)$ is closed for the topology
induced on $\F(A)$ by the pro-$G$ topology.
\end{proposition}
The proof relies on the following result (see~\cite[Theorem 5.1]{Hall1949} or \cite[Proposition 3.10]{LyndonSchupp2001}). Recall that a \emph{free factor}
\index{group!free factor}of a group $G$ is a subgroup $H$ such that $G$ is a free product
of $H$ and a subgroup $K$ of $G$. When $G$ is a free group, this is
equivalent to the following property: for some basis $X$ of $H$
there is a subset $Y$ of $G$ such that $X\cup Y$
is a basis of $G$.
\begin{theorem}[Hall]\label{theoremHall}
 For any finitely generated subgroup $H$
of $\F(A)$ and any $x\in \F(A)\setminus H$,
there is a subgroup of finite index $K$ such that $H$ is a free
factor of $K$ and
$x\notin K$ .
\end{theorem}
To deduce Proposition~\ref{propositionFGClosed} from Hall's Theorem,
consider a sequence $(x_n)$ of elements of $H$ converging to
some $x\in \F(A)$. Suppose that $x\notin H$. By Theorem~\ref{theoremHall},
there is a subgroup $K$ of finite index in $\F(A)$ containing
$H$ such that
$x\notin K$. Thus $(x_n)$ cannot converge to $x$.

It follows from Theorem~\ref{theoremHall} that one has
the following result~\cite[Corollary 2.2]{CoulboisSapirWeil2003}
\begin{corollary}\label{corollaryCoulbois}
Any injective morphism
$\varphi:\F(B)\rightarrow \F(A)$ between finitely
generated free groups extends to an injective
continuous morphism $\hat{\varphi}:\widehat{\F(B)}\rightarrow \widehat{\F(A)}$
\end{corollary} 
\begin{proof}
Let $H=\varphi(\F(B))$. Then $\varphi$ is an isomorphism between
$\F(B)$ and $H$ which extends to an isomorphism between
$\widehat{\F(B)}$ and $\widehat{H}$, the completion of $H$ with
respect to the profinite metric. But, by Proposition~\ref{propositionFGClosed},
the subgroup $H$ is closed in $\F(A)$ and thus $\widehat{H}$ is
the same as the closure $\bar{H}$ in $\widehat{FG(A)}$, which shows
that $\widehat{\varphi}$ is an injective morphism from
$\widehat{\F(B)}$ into $\widehat{\F(A)}$.
\end{proof}

\subsection{Recognizable sets}
A subset $X$ of a monoid $M$ is \emph{recognizable}\index{recognizable!set} if
it is recognized by a morphism into a finite monoid, that is, if
 there is a morphism
$\varphi:M\rightarrow N$ into a finite monoid $N$ which recognizes $X$.
We also say that $X$ is recognized by $\varphi$.
\begin{proposition}\label{propositionRecognizableClopen}
The following conditions are equivalent for a set $X\subset A^*$.
\begin{enumerate}
\item[\rm(i)] $X$ is recognizable.
\item[\rm(ii)] the closure $\bar{X}$ of $X$ in $\widehat{A^*}$
is open and $X=\bar{X}\cap A^*$.
\item[\rm(ii)] $X=K\cap A^*$ for some clopen set $K\subset \widehat{A^*}$.
\end{enumerate} 
\end{proposition}
\begin{proof}
Assume that $X$ is recognized by a morphism $\varphi:A^*\rightarrow S$
from $A^*$ into a finite monoid $S$. By the universal property of 
$\widehat{A^*}$, there is a unique continuous morphism $\hat{\varphi}$
extending $\varphi$. Then $X=\hat{\varphi}^{-1}\varphi(X)$ is open
and satisfies $X=\bar{X}\cap A^*$. Thus (i)$\Rightarrow$ (ii).
The implication (ii)$\Rightarrow$ (iii) is trivial. Finally,
assume that (iii) holds. By Proposition~\ref{propositionClopen} there
exists a continuous morphism $\psi:\widehat{A^*}\rightarrow S$
into a finite monoid $S$ which recognizes $K$. Let $\varphi$
be the restriction of $\psi$ to $A^*$. Then $X=A^*\cap K=A^*\cap \psi^{-1}\psi(K)$
and so $X$ is recognizable.
\end{proof}
As an example, the sets of the form $\widehat{A^*}w\widehat{A^*}=\overline{A^*wA^*}$
are clopen sets and so are the sets $w\widehat{A^*}=\overline{wA^*}$. This shows
that a pseudoword has a well-defined set of finite factors
and a well-defined prefix of every finite length.

The analogue of Proposition~\ref{propositionClopen} for the pro-$G$ topology
is also true (it actually holds in the pro-$V$ topology for any
pseudovariety $V$). Thus a set $X$ is recognizable by a morphim
on a finite group if and only if $X=K\cap A^*$ for some clopen set
$K\subset\widehat{\F(A)}$. In condition (ii), one has to add
that $\bar{X}\cap A^*=X$, a condition always satisfied for the
closure with respect to the pro-$M$ topology.

\subsection{The natural metric}

The \emph{natural metric}\index{profinite!monoid!natural metric} on a profinite monoid $M$ is 
defined by 
\begin{displaymath}
d(u,v)=\begin{cases}2^{-r(u,v)}&\text{if $u\ne v$}\\0&\text{otherwise}\end{cases}
\end{displaymath} where $r(u,v)$ is the minimal
cardinality of a monoid $N$ for which there is a continuous morphism
$\varphi:M\rightarrow N$ such that $\varphi(u)\ne\varphi(v)$.

It is actually an \emph{ultrametric}\index{ultrametric} since it satisfies the condition
\begin{displaymath}
d(u,w)\le\min(d(u,v),d(v,w))
\end{displaymath}
stronger than the triangle inequality.

 \begin{proposition}
For a finitely generated profinite semigroup $S$, the topology is induced
by the natural metric.
\end{proposition}
\begin{proof}
Denote $B_\varepsilon(u)=\{v\in S\mid d(u,v)<\varepsilon\}$. Let $K$ be a
clopen set in $S$. By Proposition~\ref{propositionClopen} there is a
continuous morphism $\varphi:S\rightarrow T$ into a finite semigroup $T$
which recognizes $K$. Let $n=\Card(T)$ and set $\varepsilon=2^{-n}$.
For $s\in S$ and $t\in B_\varepsilon(s)$, we have $d(s,t)<\varepsilon$
and thus $r(s,t)>n$. It implies that $\varphi(s)=\varphi(t)$.
Since $\varphi$ recognizes $K$, we conclude that $t\in K$.
Thus  the ball $B_\varepsilon(s)$ is
contained in $\varphi^{-1}(t)$ and thus $K$ is a union of open balls.
Thus, since the clopen sets form a basis of the topology, any open set is a union of open balls.

Conversely, consider the open ball $B=B_{2^{-n}}(u)$. Since there is a finite
number of isomorphism types of semigroups with at most $n$ elements, and since
$S$ is finitely generated, there
are finitely many kernels of continuous morphisms into such semigroups
and so their intersection is a clopen congruence on $S$. It follows that
there exists a continuous morphism $\varphi:S\rightarrow T$ into a finite
semigroup such that $\varphi(u)=\varphi(v)$ if and only if $r(u,v)>n$.
Hence $B=\varphi^{-1}\varphi(B)$ so that $B$ is open.
\end{proof}
This leads to an alternative definition of the free profinite monoid.
\begin{theorem}
For a finite set $A$, the completion of  $A^*$ for the natural
metric is the free profinite monoid $\widehat{A^*}$.
\end{theorem}
\paragraph{Presentations of profinite semigroups}
A congruence of a profinite semigroup is called \emph{admissible}\index{admissible congruence}
if its classes are closed and the quotient is profinite.
In other terms,  admissible congruences are the kernels of
continuous homomorphisms into profinite monoids.

Given a set $X$ and a binary relation $R$ on the monoid
$\widehat{X^*}$, the profinite semigroup $\langle X\mid R\rangle$
is the quotient of $\widehat{X^*}$ by the admissible
congruence generated by $R$. It is also said to have the
\emph{presentation}\index{monoid!presentation} $\langle X\mid R\rangle$.

The same notion holds for groups instead of semigroups
and use the notation $\langle X\mid R\rangle_S$
or $\langle X\mid R\rangle_G$ to specify if the
presentation is as a profinite semigroup or as a profinite group.

The following is \cite[Lemma 2.2]{AlmeidaCosta2013}.
\begin{proposition}\label{lemma2.2}
Let $S$ be a profinite semigroup and let $\varphi$ be an automorphism
of $S$. Let $\pi$ be a continuous homomorphism from $\widehat{A^*}$
onto $S$ and let $\Phi$ be a continuous endomorphism of $\widehat{A^*}$
such that the diagram below commutes.
\begin{figure}[hbt]
\centering
\gasset{Nframe=n,Nadjust=wh}
\begin{picture}(15,15)
\node(Sl)(0,0){$S$}\node(Sr)(15,0){$S$}
\node(A^*l)(0,15){$\widehat{A^*}$}\node(A^*r)(15,15){$\widehat{A^*}$}

\drawedge(Sl,Sr){$\varphi$}\drawedge(A^*l,A^*r){$\Phi$}
\drawedge(A^*l,Sl){$\pi$}\drawedge(A^*r,Sr){$\pi$}
\end{picture}
\end{figure}
Then $S$ has the presentation $\langle A\mid R\rangle$ with
$R=\{(\Phi^\omega(a),a)\mid a\in A\}$.
\end{proposition}

%%%%%%%%%%%%%%%%%%%%%%%%%%%%%%%%%%%%%%%%
\section{Profinite codes}\label{sectionProfiniteCodes}
We will expore the notion of a code in 
the free profinite
monoid. 
\subsection{Finite codes}
Let $A,B$ be finite alphabets.
Any morphism $\beta:B^*\rightarrow A^*$ extends uniquely by continuity
to  a continuous morphism  $\hat{\beta}:\widehat{B^*}\rightarrow \widehat{A^*}$.
A finite set $X\subset \widehat{A^*}$ is called a \emph{profinite code}\index{profinite!code}\index{code!profinite}
if the continuous extension $\hat{\beta}$
of any morphism $\beta:B^*\rightarrow A^*$
inducing a bijection from $B$ onto $X$ 
is injective.

The following statement is from~\cite{MargolisSapirWeil1998}.
\begin{theorem}\label{theoremCodeProfinite}
Any finite code $X\subset A^+$ is a profinite code.
\end{theorem}
\begin{proof} Let $\beta:B^*\rightarrow A^*$ be a coding morphism for $X$.
We have to show that for any pair $u,v\in \widehat{B^*}$
of  distinct elements, we have $\hat{\beta}(u)\ne\hat{\beta}(v)$,
that is, there
is a continuous morphism $\hat{\alpha}:\widehat{A^*}\rightarrow M$
into a finite monoid $M$
such that $\hat{\alpha}\hat{\beta}(u)\ne\hat{\alpha}\hat{\beta}(v)$.
For this, let $\psi:\widehat{B^*}\rightarrow N$ be a continuous morphism
into a finite monoid $N$ such that $\psi(u)\ne\psi(v)$.
Let $P$ be the set of proper prefixes of $X$ and 
let $\T$ be the prefix transducer associated to $\beta$
(see~\cite{BerstelPerrinReutenauer2009}). Let
$\alpha$ be the morphism from $A^*$ into the monoid of $P\times P$-matrices
with elements in $N\cup 0$ defined as follows. For $x\in A^*$ and $p,q\in P$,
we have 
\begin{displaymath}
\alpha(x)_{p,q}=\begin{cases}\psi(y)&\text{if there is a path $p\edge{x|y}q$}\\
0&\text{otherwise}.
\end{cases}
\end{displaymath}
Then $M=\alpha(A^*)$ is a finite monoid and 
$\alpha$ extends to a continuous morphism $\hat{\alpha}:\widehat{A^*}
\rightarrow M$.
Since, by \cite[Proposition 4.3.2]{BerstelPerrinReutenauer2009}, the transducer $\T$
realizes the decoding function of $X$,
we have $\alpha\beta(y)_{1,1}=\psi(y)$ for any $y\in B^*$. By continuity, we have
$\hat{\alpha}\hat{\beta}(y)_{1,1}=\psi(y)$
for any $y\in \widehat{B^*}$. Then $\hat{\alpha}$
is such that  $\hat{\alpha}\hat{\beta}(u)\ne\hat{\alpha}\hat{\beta}(v)$. Indeed 
$\hat{\alpha}\hat{\beta}(u)_{1,1}=\psi(u)\ne\psi(v)=\hat{\alpha}\hat{\beta}(v)_{1,1}$.
\end{proof}
\begin{example}
Let $A=\{a,b\}$ and $X=\{a,ab,bb\}$. The set $X$ is a suffix code. It has infinite
deciphering delay since $abb\cdots=a(bb)(bb)\cdots=(ab)(bb)(bb)\cdots$. Nonetheless,
$X$ is a profinite code in agreement with Theorem~\ref{theoremCodeProfinite}.
Note that, in $\widehat{A^*}$, the pseudowords $a(bb)^\omega$ and $ab(bb)^\omega$
are distinct (the first one is a limit of words of odd length and the
second one of words of even length).
\end{example}
Theorem~\ref{theoremCodeProfinite} shows that the
closure of the submonoid generated by a finite code
is a free profinite monoid. This has been extended to rational 
codes in~\cite{AlmeidaSteinberg2009}. Actually, for any
rational code $X$, the profinite submonoid generated by $X$
is free with basis the closure $\bar{X}$ of $X$
\cite[Corollary 5.7]{AlmeidaSteinberg2009}.
Note that we have here a free profinite monoid with infinite
basis (see~\cite{AlmeidaSteinberg2009} for an extension of the
notion of free profinite monoid to infinite alphabets).
\begin{example}
Let $A=\{a,b\}$ and $X=a^*b$. Then the profinite monoid $\widehat{X^*}$ is free
with basis the uncoutable set $\bar{X}=\{a^v b\mid v\in\hat{\N}\}$.
\end{example}
\subsection{Codes and submonoids}
A submonoid $N$ of a monoid $M$
 is called \emph{stable}\index{stable}\index{submonoid!stable}
if for any $u,v,w\in M$, whenever $u,vw,uv,w\in N$
then $v\in N$. It is called \emph{right unitary}\index{right unitary}
if for every $u,v\in A^*$, $u,uv\in N$ implies $v\in N$.
It is well known that a submonoid of $A^*$ is stable
if and only if it is generated by a code
and it is unitary if and only if it is generated by a prefix code.

The following statement extends these notions to pseudowords.
\begin{proposition}\label{propositionUnitary}
Let $N$ be a recognizable submonoid of $A^*$. If $N$
is stable (resp. unitary) its closure 
in $\widehat{A^*}$ is a stable (resp. unitary)
submonoid of $\widehat{A^*}$.
\end{proposition}
\begin{proof}
The set $\bar{N}$ is clearly a submonoid of $\widehat{A^*}$.
Assume that $N$ is stable.
Let $u,v,w\in \widehat{A^*}$ be such that $uv,w,u,vw\in \bar{N}$.
Let $(u_n),(v_n)$ and $(w_n)$ be  sequences of
words  converging   to $u,v$ and $w$ respectively
with $u_n,w_n\in N$. By Proposition~\ref{propositionRecognizableClopen},
the closure
$\bar{N}$ of the submonoid $N$ is open in $\widehat{A^*}$ and 
$N=\bar{N}\cap A^*$. Since $\bar{N}$
is open, we have $u_nv_n,v_nw_n\in\bar{N}$
for large enough $n$ and thus also
$u_nv_n,v_nw_n\in N$. Since $N$ is stable, this implies $v_n\in N$
for large enough $n$, which implies $v\in\bar{N}$.
Thus $\bar{N}$ is stable.
The proof when $N$ is unitary is similar.
\end{proof}
When $X$ is a prefix code, every word $w$ can be written in a unique way
$w=xp$ with $x\in X^*$ and $p\in A^*\setminus XA^*$, a word without any prefix in $X$. When $X$ is a maximal prefix code, a word which has no prefix
in $X$ is a proper prefix of a word in $X$ and
thus every word $w$ can be written in a unique way $w=xp$ with
$x\in X^*$ and $p$ a proper prefix of $X$.
This extends to pseudowords as follows.
\begin{proposition}\label{propositionUnitary2}
Let $X\subset A^+$ be a finite maximal
prefix code and let $P$ be its set of proper prefixes.
 Any pseudoword $w\in\widehat{A^*}$
has a unique factorization $w=xp$ with $x\in \overline{X^*}$ and $p\in P$.
\end{proposition}
\begin{proof}
Let $(w_n)$ be a sequence of words converging to $w$. For each $n$,
we have $w_n=x_np_n$ with $x_n\in X^*$ and $p_n\in P$. Taking a subsequence,
we may assume that the sequences $(x_n)$ and $(p_n)$ converge to $x\in\overline{X^*}$
and $p\in\bar{P}$. Since $X$ is finite, $P$ is finite
and thus $\bar{P}=P$. This proves the existence of a factorization.
To prove the uniqueness, consider $xp=x'p'$ with $x,x'\in \widehat{X^*}$
and $p,p'\in  P$. Then, assuming that $p$ is longer than $p'$, we
have $p=up'$ and $xu=x'$ for some $u\in A^*$. Since $\overline{X^*}$ is unitary, we have
$u\in X^*$ and thus $u=\varepsilon$.
\end{proof}
%%%%%%%%%%%%%%%%%%%%%%%%
\section{Uniformly recurrent sets}\label{factorialSets}
In this section, we study the closure in the free profinite
monoid of a uniformly recurrent set.
\subsection{Uniformly recurrent pseudowords}
Let $A$ be a finite alphabet. A set of finite words on $A$ is \emph{factorial}
\index{factorial set}if it contains the alphabet $A$ and all the factors of its elements.

A factorial set $F$ is \emph{recurrent}\index{recurrent set}\index{factorial set!recurrent} if for
any $x,z\in F$ there is some $y\in F$ such that $xyz\in F$.

Recall that an infinite factorial set $F$ of finite words is said to
be \emph{uniformly recurrent}\index{uniformly recurrent set} or \emph{minimal}
\index{minimal set}\index{factorial set!minimal} if for any $x\in F$ there
is an integer $n\ge 1$ such that $x$ is a factor of every
word in $F$ of length $n$.

A uniformly recurrent set is obviously recurrent.

Note that a uniformly recurrent set is actually minimal for inclusion among
the infinite factorial sets. Indeed, assume first that $F$ is uniformly recurrent. Let $F'\subset F$ be an infinite factorial set. Let $x\in F$ and let
$n$ be such that $x$ is factor of any word of $F$ of length $n$. Since $F'$
is factorial infinite, it contains a word $y$ of length $n$. Since
$x$ is a factor of $y$, it is in $F'$. Thus $F$ is mimimal.
Conversely consider $x\in F$ and let $T$ be the set of words in $F$
which do not contain $x$ as a factor. Then $T$ is factorial. Since
$F$ is minimal among infinite factorial sets, $T$ is finite.
Thus there is an $n$ such that $x$ is a factor of all words of $F$ of length $n$.

A factorial set $F$ is \emph{periodic}\index{periodic set}\index{factorial set!periodic}\index{set!periodic} if it is the set of factors of a finite word $w$. One may always
assume $w$ to be primitive, that is not a power of another word.
In this case, the length of $w$ is called the \emph{period}\index{period}
 of $F$.
A periodic set is obviously uniformly recurrent.

For an infinite pseudoword $w$, we denote by $F(w)$ the set of finite
factors of $w$. It is an infinite factorial set.

An infinite pseudoword $w$ is \emph{uniformly recurrent}\index{uniformly recurrent!pseudoword} if 
$F(w)$ is uniformly recurrent. This is the same definition as
the definition commonly used for infinite words.

The following is~\cite[Lemma 2.2]{Almeida2005b}. 
\begin{proposition}\label{propositionTrivial}
An infinite pseudoword is uniformly recurrent if and only if
all its infinite factors have the same finite factors.
\end{proposition}
\begin{proof}
Let $w$ be an infinite pseudoword. Assume first that $w$ is
uniformly recurrent. Let $u$ be an infinite factor of $w$. Let
us show that $F(u)=F(w)$. The inclusion $F(u)\subset F(w)$ is clear.
Conversely, let $x\in F(w)$. Since $w$ is
uniformly recurrent, $x$ is a factor of every long enough finite factor of $w$.
In particular, $x$ is a factor of every long enough finite prefix
of $u$. Thus $x\in F(u)$.

Conversely, assume that $F(w)=F(u)$ for all infinite factors $u$ of $w$.
Let $v$ be a finite factor of $w$. Arguing by contradiction, assume that
there are arbitrary long factors of $w$ which do not have $v$ as a factor.
This infinite set contains, by Proposition~\ref{propositionFactors},
 a subsequence converging to some infinite pseudoword $u$ which is also a factor
of $w$, and such that $v\not\in F(u)$, a contradiction.
\end{proof}

 Recall that
the $\JJ$-order in a monoid $M$ is defined by $x\le_\JJ y$ if 
$x$ is a factor of $y$. Two elements $x,y$ are $\JJ$-equivalent if
each one is a factor of the other (this is one of the \emph{Green's relations}
\index{Greene's relations}, see
\cite{BerstelPerrinReutenauer2009}). 

Replacing the notion of factor by prefix (resp. suffix), one obtains
the $\RR$-order (resp. $\LL$-order).
Thus, two elements $x,y$ of a monoid $M$ are $\RR$-equivalent (resp. $\LL$-equivalent)
if 
$xM=yM$ (resp. $Mx=My$). The $\HH$-equivalence is the intersection of
$\RR$ and $\LL$. In any monoid, one has $\RR\LL=\LL\RR$ and 
one denotes $\DD$ the equivalence $\RR\LL=\LL\RR$ which is the supremum
of $\RR$ and $\LL$. In a compact monoid, one has $\DD=\JJ$. The proof is
the same as in a finite monoid. It uses the fact that a compact
monoid satisfies the \emph{stability condition}\index{stability condition}: if $x\le_\RR y$
and $x\JJ y$, then $x\RR y$ and dually for $\LL$.

The following result is~\cite[Theorem 2.6]{Almeida2005b}. It gives an algebraic characterization
of uniform recurrence in the free profinite monoid.
\begin{theorem}\label{theoremUniformRecurrence}
An infinite pseudoword is uniformly recurrent
if and only if it is $\JJ$-maximal.
\end{theorem}
The proof uses three lemmas. An element $s$ of a semigroup $S$
is \emph{regular}\index{regular!element} if there is some $x\in S$ such that $sxs=s$.
In a compact semigroup,  a $\JJ$-class contains a regular
element if and only if all its elements are regular,
if and only if it contains an idempotent.

For a pseudoword $w$, we denote $X(w)$ the  set of all
infinite pseudowords which are limits of sequences
of finite factors of $w$.
\begin{lemma}\label{lemma2.3}
Let $w$ be uniformly recurrent pseudoword over a finite
alphabet $A$.
\begin{enumerate}
\item Every element of $X(w)$ is a factor of $w$.
\item All elements of $X(w)$ lie in the same $\JJ$-class of $\widehat{A^*}$.
\item Every element of $X(w)$ is regular.
\end{enumerate} 
\end{lemma}
\begin{proof}
%Assertion 1 results directly from Proposition~\ref{propositionFactors}.
Assertion 1 results from the fact that $F(w)$ is closed.

Suppose that $u,v\in X(w)$. By Proposition~\ref{propositionTrivial}, they have the same set of finite factors. Thus, by Assertion 1, they are $\JJ$-equivalent.

Assume that $u$ is the limit of a sequence $(u_n)_{n\ge 0}$ of finite factors of $w$. Since $w$ is recurrent, there are finite words $v_n$ such that $u_nv_nu_n$
is a factor of $w$. If $v$ is an accumulation point of the sequence $v_n$,
then $uvu$ is a factor of $w$ which belongs to $X(w)$. By Assertion 2,
it is $\JJ$-equivalent to $u$. In a compact monoid, by the stability
condition, this implies that $u$ and $uvu$ are $\HH$-equivalent
and thus that $u$ is regular (indeed, $u\HH uvu$ implies $u\HH u(vu)^\omega$
and thus $(vu)^\omega$ is an idempotent in $J(u)$).
\end{proof}

\begin{lemma}
Let $w$ be a uniformly recurrent pseudoword.
Each $\HH$-class contained in the $\JJ$-class of $w$
contains some element of $X(w)$.
\end{lemma}
\begin{proof}
Let $u\in J(w)$. Denote by $x_n$ and $y_n$ the prefix and the suffix of $u$
of length $n$. Since $w$ is uniformly recurrent by Proposition~\ref{propositionTrivial}, there is a factor $t_n$
of $w$ of length at least $2n$ having $x_n$ as a prefix and $y_n$ as a suffix.
Taking a subsequence, we may assume that the sequences $(x_n)$, $(y_n)$
and $(t_n)$ converge to $x,y,t$. 
Then $x,y,t\in J(w)$ by Lemma~\ref{lemma2.3}(2). Since $x\ge_\RR t$
and $t\le_\LL y$, by stability, we obtain $u\HH t$.\end{proof}

\begin{lemma}\label{lemmaR}
Let $u$ be a uniformly recurrent pseudoword and suppose $v$ is a pseudoword
 such that $uv$ is still uniformly recurrent. Then $u$ and $uv$ are $\RR$-equivalent.
\end{lemma}
\begin{proof}
Suppose first that $v$ is finite.
Let $u_n$ be the suffix of $u$ of length $n$. Since $u$ is
an infinite factor of $uv$ which is uniformly recurrent, they
have the same finite factors by Lemma
\ref{propositionTrivial}. Hence for every $n$ there is some
$m(n)$ such that $u_{m(n)}=x_nu_nvy_n$. By compactness, we may assume
by taking subsequences that the sequences $x_n,y_n,u_n$ converge
to $x,y,u'$ respectively. Then by continuity of the multiplication, 
the sequence $u_{m(n)}$ converges to $xu'vy$. Since the limits
of two convergent sequences of suffixes of the same pseudoword
are $\LL$-equivalent, we obtain that $xu'vy\LL u$ and thus
$u\RR u'v$ by stability. Since $u'$ is the limit of a sequence
of suffixes of $u$, there is some factorization of the
form $u=zu'$. Since the $\RR$-equivalence is a left congruence,
we finally obtain $uv=zu'v\RR zu'=u$.

Assume next that $v$ is infinite. We assume by contradiction that
$u>_\RR uv$. Let $v_n$ be a sequence of finite
words converging to $v$. Taking a subsequence, we may assume that
$uv_n>_\RR u$ for all $n$. Thus for each $n$, we have a factorization
$v_n=x_na_ny_n$ with $a_n\in A$ such that $u\RR ux_n>_\RR ux_na_n\ge_\RR uv$.
Since the alphabet is finite and $\widehat{A^*}$ is compact
we may, up to taking a subsequence, assume that the letter
sequence $a_n$ is constant and the sequences $x_n,y_n$ converge
to $x$ and $y$ respectively.
Thus we have $p=xay$ with $a\in A$ and $u\RR ux>_\RR uxa\ge_ \RR uv$.
On the other hand, since $ux$ and $uxa$ are infinite factors
of $uv$, they are both uniformly recurrent by Proposition~\ref{propositionTrivial}. By the first part, we have $ux\RR uxa$, a contradiction.
\end{proof}
A dual result holds for the $\LL$-order.

\begin{proofof}{of Theorem~\ref{theoremUniformRecurrence}}
Suppose first that $w$ is $\JJ$-maximal as an infinite pseudoword.
If $v$ is an infinite factor of $w$, it is $\JJ$-equivalent to $w$.
Hence $v,w$ have the same factors and, in particular, the same finite
factors. By Proposition~\ref{propositionTrivial}, $w$ is uniformly recurrent.

Suppose conversely that $u,w\in\widehat{A^*}\setminus A^*$ 
are such that  $u\ge_\JJ w$ with $w$  uniformly recurrent.
 Set $w=puq$ with $p,q\in\widehat{A^*}$.  
By the dual of Lemma~\ref{lemmaR}, we have $pu\LL u$. And
by Lemma \ref{lemmaR}, we have $pu\RR puq$. Thus $u$ and $w$
are $\JJ$-equivalent.
\end{proofof}
\begin{example}
The $\JJ$-class of $a^\omega$ in $\widehat{A^*}$ is made of one $\HH$-class.
The $\JJ$-class of $(ab)^\omega$ has four $\HH$-classes. It is represented in Figure~\ref{FigureJclassab}.
\begin{figure}[hbt]
\centering
\begin{picture}(30,30)(0,-8)
$
\def\rb{\hspace{2pt}\raisebox{0.8ex}{*}}\def\vh{\vphantom{\biggl(}}
    \begin{array}%
    {r|@{}l@{}c|@{}l@{}c|}%
    %\multicolumn{1}{r}{}&\multicolumn{2}{c}{3}&\multicolumn{2}{c}{1}\\
    \cline{2-5}
    & \vh &(ab)^\omega &\vh  &(ab)^\omega a \\
    \cline{2-5}
    &\vh&b(ab)^\omega &\vh&(ba)^\omega\\
    \cline{2-5}
    \end{array}
$
\end{picture}
\caption{The $\JJ$-class of $(ab)^\omega$.}\label{FigureJclassab}
\end{figure}
\end{example}

\subsection{The $\JJ$-class $J(F)$}
%%%%%%%%%%%%%%%%%%%%%%%%%%%%%%%%
%\section{Factorial, neutral and tree sets}\label{factorialSets}
Let $F$ be a uniformly recurrent set of finite words on the alphabet $A$. 
The closure $\bar{F}$ of $F$ in $\widehat{A^*}$ is also factorial
(see~\cite[Proposition 2.4]{AlmeidaCosta2009}). 
The proof relies on the following useful lemma from~\cite[Lemma 2.5]{AlmeidaCosta2009}.
\begin{lemma}\label{lemmaUseful}
For every $u,v\in\widehat{A^*}$ and every sequence $(w_n)$ converging to $uv$,
there are sequences $(u_n),(v_n)$ such that $\lim u_n=u$, $\lim v_n=v$
and $(u_nv_n)$ is a subsequence of $(w_n)$.
\end{lemma}

The set of two-sided infinite words with all their factors in $F$
is denoted $X(F)$. It is closed for the product topology of $A^\Z$.
It is also invariant by the shift $\sigma:A^\Z\rightarrow A^\Z$
defined by $y=\sigma(x)$ if $y_n=x_{n+1}$ for any $n\in \Z$.
Such a closed and shift invariant set is called a \emph{subshift}\index{subshift}.
It is classical that a subshift is of the form $X(F)$ for some
uniformly recurrent set if and only if it is minimal (see
\cite{LindMarcus1995} for example).

By the results of Section~\ref{factorialSets}, all
the infinite pseudowords in the closure $\bar{F}$ of $F$ 
are $\JJ$-equivalent. We denote by $J(F)$ their $\JJ$-class.
\begin{example}\label{exampleFibonacciSet}
The Fibonacci morphism $\varphi$ is primitive. The set $F$ of factors 
of the words $\varphi^n(a)$ for $n\ge 1$ is
called the \emph{Fibonacci set}\index{Fibonacci!set}. It contains the infinitely recurrent
pseudowords $\varphi^\omega(a)$ and $\varphi^\omega(b)$.
\end{example}

For a uniformly recurrent set $F$, the $\JJ$-class $J(F)$ can be described
as follows. For $x\in \widehat{A^*}$, denote by $\stackrel{\rightarrow}{x}$
the right infinite word whose finite prefixes are those of $x$.
Symmetrically, $\stackrel{\leftarrow}{x}$ is the left
inifnite word whose finite suffixes are those of $x$.
The following is~\cite[Lemma 6.6]{AlmeidaCosta2009}.
\begin{proposition}
For a uniformly recurrent set $F$, two words $u,v\in J(F)$ are
$\RR$-equivalent if and only if $\stackrel{\rightarrow}{u}=\stackrel{\rightarrow}{v}$ and $\LL$-equivalent if and only if $\stackrel{\leftarrow}{u}=\stackrel{\leftarrow}{v}$
\end{proposition}
It follows from this that $w\in J(F)$ belongs to a subgroup if and only
if the two-sided infinite word $\stackrel{\leftarrow}{w}\cdot\stackrel{\rightarrow}{w}$ has all its factors in $F$. Indeed, the finite factors of $w^2$
are those of $w$ plus the the products $uv$ where $u$ is a finite
suffix of $w$ and $v$ is a finite prefix of $w$ \cite[Lemma 8.2]{AlmeidaCosta2009}.

Thus the maximal subgroups of $J(F)$ are in bijection with the elements
of the set $X(F)$ of two-sided infinite words with all their factors in $F$.
For $x\in X(F)$, we denote by $H_x$ the maximal subgroup corresponding
to $x$.
\begin{example}
Let $F$ be the Fibonacci set and let $w=\varphi^\omega(a)$. The right infinite
word $\stackrel{\rightarrow}{w}$ is the Fibonacci word. The left infinite
word $\stackrel{\leftarrow}{w}$ is the word with suffixes 
$\varphi^{2n}(a)$ (see Section~\ref{sectionFibonacci}). The two
sided infinite word $\stackrel{\leftarrow}{w}\cdot\stackrel{\rightarrow}{w}$
is a fixed point of $\varphi^2$.
\end{example}

\begin{example}\label{exampleMorphismAC}
let $A=\{a,b\}$ and let $\varphi:A^*\rightarrow A^*$ be defined by
$\varphi(a)=ab$ and $\varphi(b)=a^3b$. Let $F$ be the set
of factors of $\varphi^\omega(a)$. Since $\varphi$ is primitive,
$F$ is uniformly recurrent. 
The pseudowords $\varphi^\omega(a)$ and $\varphi^\omega(b)$
belong to the same $\HH$-class of $J(F)$. Indeed, we have
$\stackrel{\rightarrow}{\varphi^\omega(a)}=\stackrel{\rightarrow}{\varphi^\omega(b)}$
and $\stackrel{\leftarrow}{\varphi^\omega(a)}=\stackrel{\leftarrow}{\varphi^\omega(b)}$.
\end{example}
%%%%%%%%%%%%%%%%%%%%%%%%
\subsection{Fixed points of substitutions}\label{sectionSustitutions}
Let $A$ be a finite alphabet.
Let $\varphi:A^*\rightarrow A^*$ be a morphism, also called a \emph{substitution}\index{substitution} over $A$. Then $\varphi$ extends uniquely by continuity
to a morphism still denoted $\varphi:\widehat{A^*}\rightarrow \widehat{A^*}$.
The monoid $\End(\widehat{A^*})$ is profinite by Theorem~\ref{theoremEndomorphismMonoid}. Thus the morphism $\varphi^\omega$ is  well defined as the unique idempotent in the closure of the semigroup generated by $\varphi$.

%A substitution $\varphi:A^*\rightarrow A^*$ is said to be 
%\emph{weakly primitive} if there is an ineger $n\ge 1$ such that 
%$\varphi^n(a)$ has the same set of factors of length $2$ for each 
%$a\in A$ and this set is not contained in $A$. 

A substitution $\varphi:A^*\rightarrow A^*$ is said to be \emph{primitive}\index{primitive}\index{substitution!primitive}\index{morphism!primitive} if there is an integer $n$ such
that all letters appear in every $\varphi^n(a)$ for $a\in A$. 
%It can be verified that a primitive substitution is wekly primitive 
%(this follows from the fact
%that any fixed point of a primitive substitution is uniformly recurrent).

A fixed point of a substitution $\varphi$ is an infinite word
$x\in A^\N$ such that $\varphi(x)=x$.
As well known, a fixed point of a primitive substitution is uniformly
recurrent (see~\cite[Proposition 1.3.2]{PytheasFogg2002} for example).

The following is a particular case of~\cite[Theorem 3.7]{Almeida2005b}
(in which the notion weakly primitive substitution is introduced).

\begin{theorem}
Let $\varphi$ be a primitive
substitution over a finite alphabet $A$. Then the 
 pseudowords $\varphi^\omega(a)$ with $a\in A$
are uniformly recurrent and are all $\JJ$-equivalent.
\end{theorem}
\begin{proof}
Let $u$ be a finite factor of $\varphi^\omega(a)$ for some $a\in A$.
Then there is an integer $N$ such that
$u$ is a factor of any factor of length $N$ of $\varphi^n(b)$ for all $b\in A$.
Thus $u$ is a factor of any finite factor of  length $N$ of any
$\varphi^\omega(b)$. This proves both claims.
\end{proof}

\begin{example}
The Fibonacci morphism\index{Fibonacci!morphism} $\varphi:a\mapsto ab,b\mapsto a$ is primitive.
Thus the pseudowords $\varphi^\omega(a)$ and $\varphi^\omega(b)$ are
uniformly recurrent and $\JJ$-equivalent.
\end{example}
\begin{example}\label{exampleThueMorse}
The \emph{Thue-Morse substitution}\index{substitution!Thue-Morse} is the morphism $\tau:a\mapsto ab,b\mapsto ba$.
It is primitive. The unique fixed point $x=abbabaab\cdots$ of $\tau$ beginning with $a$ is
called the \emph{Thue-Morse infinite word}. The set of its factors is called
the \emph{Thue-Morse set}.
\end{example}
Given a substitution $\varphi$ over $A$, we denote by $\varphi_G$
the endomorphism of $\widehat{\F(A)}$ such that the following diagram
commutes 
\begin{figure}[hbt]
\centering
\gasset{Nframe=n,Nadjust=wh}
\begin{picture}(20,15)
\node(FGAl)(0,0){$\widehat{\F(A)}$}\node(FGAr)(20,0){$\widehat{\F(A)}$}
\node(A^*l)(0,15){$\widehat{A^*}$}\node(A^*r)(20,15){$\widehat{A^*}$}

\drawedge(FGAl,FGAr){$\varphi_G$}\drawedge(A^*l,A^*r){$\varphi$}
\drawedge(A^*l,FGAl){$\pi$}\drawedge(A^*r,FGAr){$\pi$}
\end{picture}
\end{figure}
where $\pi:\widehat{A^*}\rightarrow \widehat{\F(A)}$ denotes the canonical projection.

For a substitution $\varphi$ over $A$, the endomorphism $\varphi_G^\omega$
is the identity if and only if $\varphi$ is invertible (as a map from
$\F(A)$ into itself). 

Indeed, if $\varphi$ is an automorphism of
$\F(A)$, then its extension to $\widehat{\F(A)}$ is also an automorphism.
But then $\varphi_G^\omega$ is the identity since in a group
one has $x^\omega=1$ for any element $x$.

Conversely, if $\varphi_G^\omega$ is the identity, then $\varphi$ is
a bijection from $\F(A)$ onto itself and thus it is an
automorphism of $\F(A)$.
\begin{example}\label{exampleFibonacciPseudo}
Let $A= \{a,b\}$ and let $\varphi:a\mapsto ab,b\mapsto a$ be the Fibonacci morphism. Then $\varphi$ is an automorphism of $\F(A)$ since $\varphi^{-1}:a\mapsto b, b\mapsto b^{-1}a$. Accordingly $\varphi_ G^\omega$ is the identity.
In particular, one has $\varphi_ G^\omega(a)=a$.
This explains in a simple way that the length of $\varphi_G^\omega(a)$
is equal to $1$ (see Section~\ref{sectionFibonacci}).
\end{example}
%%%%%%%%%%%%%%%%

%%%%%%%%%%%%%%%%%%%%%%%%%%%%
\section{Sturmian sets and tree sets}\label{sectionSturmian}

Let $F$ be a factorial set on the alphabet $A$.
For $w\in F$, we denote
\begin{eqnarray*}
L_F(w) & = & \{ a\in A \mid aw\in F\}, \\
R_F(w) & = & \{ a\in A \mid wa\in F\}, \\
E_F(w) & = & \{ (a,b) \in A \times A \mid awb \in F\}
\end{eqnarray*}
and further
\begin{displaymath}
\ell_F(w)=\Card(L_F(w)),\quad r_F(w)=\Card(R_F(w)),\quad e_F(w)=\Card(E_F(w)).
\end{displaymath}
For $w\in F$, we denote 
$$
m_F(w) = e_F(w) - \ell_F(w) - r_F(w) + 1.
$$

A word $w$ is called \emph{neutral}\index{neutral!word} if $m_F(w) = 0$.
A factorial set $F$ is \emph{neutral}\index{neutral!set} if every word in $F$ is neutral.
\begin{example}
The Fibonacci set is neutral as any Sturmian set.
\end{example}
\begin{example}
The Thue-Morse set $T$ is not neutral. Indeed, since $A^2\subset T$, one has
$m_T(\varepsilon)=1$.
\end{example}
\begin{example}\label{exampleMorphismAC2}
Let $\varphi:a\mapsto ab,b\mapsto a^3b$ be as in Example~\ref{exampleMorphismAC}. Let $F$ be the set of factors of $\varphi^\omega(a)$.
It is not neutral since
$m(a)=1$ and $m(aa)=-1$. 
\end{example}
A neutral set has complexity $kn+1$ where $k=\Card(A)-1$ (see~\cite{BertheDeFeliceDolceLeroyPerrinReutenauerRindone2013a}).

%%%%%%%%%%%%%%%%%%%%%%%%%
\subsection{Sturmian sets}
We recall here some notions concerning episturmian words (see~\cite{BerstelDeFelicePerrinReutenauerRindone2012}
for more
details and references). 

 A word $w$ is
\emph{right-special}\index{right-special word}
(resp. \emph{left-special}\index{left-special word}) if $\ell_F(w)\ge 2$ (resp.
$r_F(w)\ge 2$). A right-special (resp.
left-special) word $w$ is \emph{strict}\index{strict!right-special
  word}\index{right-special word!strict}\index{strict!left-special word}
\index{left-special word!strict} if $\ell_F(w)=\Card(A)$
(resp. $r_F(w)=\Card(A)$). In the case of a $2$-letter alphabet,
all special words are strict.

By definition, an infinite word $x$ is
\emph{episturmian}\index{episturmian word} if $F(x)$ is
closed under reversal and if $F(x)$ contains, for each $n\ge 1$, at
most one word  of length $n$ which is right-special. 

Since $F(x)$ is closed under reversal, the reversal of a right-special
factor of length $n$ is left-special, and it is the only left-special
factor of length $n$ of $x$. A suffix of a right-special factor is
again right-special. Symmetrically, a prefix of a left-special factor
is again left-special.

As a particular case, a \emph{strict}%
\index{strict!episturmian word}\index{episturmian word!strict}
episturmian word is an episturmian word $x$ with the two following
properties: $x$ has exactly one right-special factor of each length
and moreover each right-special factor $u$ of $x$ is strict, that is
satisfies the inclusion $uA\subset F(x)$
(see~\cite{DroubayJustinPirillo2001}).

 For $a\in A$, denote by
$\psi_a$ the morphism  of $A^*$ into itself, called
\emph{elementary morphism}\index{elementary
  morphism}\index{morphism!elementary}, defined by
\begin{displaymath}
\psi_a(b)=\begin{cases}ab&\text{if $b\ne a$}\\
                       a&\text{otherwise}
         \end{cases}
\end{displaymath}
Let $\psi:A^*\rightarrow \End(A^*)$ be the morphism from $A^*$ into
the monoid of endomorphisms of $A^*$ which maps each $a\in A$ to
$\psi_a$. For $u\in A^*$, we denote by $\psi_u$ the image of $u$ by
the morphism $\psi$.  Thus, for three words $u,v,w$, we have
$\psi_{uv}(w)=\psi_u(\psi_v(w))$.
A \emph{palindrome}\index{palindrome word} is a word $w$ which is
equal to its reversal.  Given a word $w$, we denote by $w^{(+)}$ the
\emph{palindromic closure}\index{palindromic closure} of $w$. It is,
by definition, the shortest palindrome which has $w$ as a prefix.

The \emph{iterated palindromic closure}\index{iterated palindromic
  closure} of a word $w$ is the word $\Pal(w)$ defined recursively as
follows. One has $\Pal(1)=1$ and for $u\in A^*$ and $a\in A$, one has
$\Pal(ua)=(\Pal(u)a)^{(+)}$. Since $\Pal(u)$ is a proper prefix of
$\Pal(ua)$, it makes sense to define the iterated palindromic closure
of an infinite word $x$ as the infinite word which is the limit of the iterated palindromic closure
of the prefixes of $x$.

\emph{Justin's Formula}\index{Justin's Formula} is the following. For
every words $u$ and $v$, one has
\begin{displaymath}
  \Pal(uv)=\psi_u(\Pal(v))\Pal(u)\,.
\end{displaymath}
This formula extends to infinite words: if $u$ is a word and $v$ is an
infinite word, then
\begin{equation}\label{JustinInfini}
  \Pal(uv)=\psi_u(\Pal(v))\,.
\end{equation}
There is a precise combinatorial description of standard episturmian
words (see e.g.~\cite{JustinVuillon2000,GlenJustin2009}). 

\begin{theorem}
  An infinite word $s$ is a standard episturmian word if and only if
  there exists an infinite word $\Delta=a_0a_1\cdots$, where the $a_n$
  are letters, such that
  \begin{displaymath}
    s=\lim_{n\to\infty} u_n\,,
  \end{displaymath}
  where the sequence $(u_n)_{n\ge 0}$ is defined by $u_n=\Pal(a_0a_1\cdots
  a_{n-1})$.  Moreover, the word $s$ is episturmian strict if and only
  if every letter appears infinitely often in~$\Delta$.
\end{theorem}

\noindent The infinite word $\Delta$ is called the \emph{directive
  word}\index{directive word} of the standard word $s$. The
description of the infinite word $s$ can be rephrased by the equation
\begin{displaymath}
  s=\Pal(\Delta)\,.
\end{displaymath}
As a particular case of Justin's
Formula, one has
\begin{equation}\label{EquationMagique}
  u_{n+1}=\psi_{a_0\cdots a_{n-1}}(a_n)u_n\,.
\end{equation}
The words $u_n$ are the only prefixes of $s$ which are palindromes. 
\begin{example}\label{directiveFibo}
The Fibonacci word is a standard episturmian word with directive word
$\Delta=ababa\cdots$. Indeed, by Formula~\eqref{EquationMagique} one has
$\Pal(\Delta)=\psi_{ab}(\Pal(\Delta))$. Since $\psi_{ab}=\varphi^2$
where $\varphi$ is the Fibonacci morphism, we have 
$\Pal(\Delta)=\varphi^2(\Pal(\Delta))$. This shows that
$\Pal(\Delta)$ is the Fibonacci word.
\end{example}
We note that for $n\ge 1$, one has
\begin{equation}
|u_{n+1}|\le 2|u_n|.\label{ineqlg}
\end{equation}
Indeed, set $u_n=u'_nba$. If $a_n=a$, the word $u'_nbaab\tilde{u}'_n$
is palindrome of length at most $2|u_n|$. If $a_n=b$, then
$u'_nbab\tilde{u'_n}$ is a palindrome of length strictly less than
$2|u_n|$.

\begin{example}
As a consequence of Equation~\eqref{EquationMagique}, when $s$ is the Fibonacci word and $\varphi$ the Fibonacci
morphism, we have for every $n\ge 0$
\begin{equation}
u_{n+1}=\varphi^{n}(a)u_n.\label{eqMagique}
\end{equation}
In view of Equation~\eqref{EquationMagique}, we need to show that
$\varphi^n(a)=\psi_{a_0\cdots a_{n-1}}(a_n)$. By Example~\ref{directiveFibo},
the directive word of $s$ is $ababa\cdots$. If $n$ is even,
then $\psi_{a_0\cdots a_{n-1}}(a_n)=\psi_{(ab)^{n/2}}(a)=\varphi^n(a)$
since $\psi_{ab}=\varphi^2$.
If $n$ is odd, then $\psi_{a_0\cdots a_{n-1}}(a_n)=\psi_{(ab)^{(n-1)/2}a}(b)
=\varphi^{n-1}(ab)=\varphi^n(a)$ and the property is true also.

As a consequence, we have in the prefix ordering for every $n\ge 0$,
\begin{equation}
u_{n}<\varphi^{n+1}(a).\label{eqInegLg}
\end{equation}
Indeed, both words are prefixes of the Fibonacci word and it is
enough to compare their lengths. For $n=0$, we have $|u_0|=0$ and $|\varphi(a)|=2$. Next, for $n\ge 1$,
we have by \eqref{eqMagique}, $u_{n}=\varphi^{n-1}(a)u_{n-1}$. Arguing by
induction, we obtain $|u_n|<|\varphi^{n-1}(a)|+|\varphi^{n}(a)|=|\varphi^{n+1}(a)|$.
\end{example}
%%%%%%%%%%%%%%%%%%%%%%%%%%
\subsection{Tree sets}
Let $F$ be a factorial set of words.
For $w \in F$, we consider the set $E_F(w)$ as an undirected graph on the set of vertices which is the disjoint union of $L_F(w)$ and $R_F(w)$ with edges the pairs $(a,b) \in E_F(w)$.
This graph is called the \emph{extension graph}\index{extension graph} of $w$.

A factorial set $F$ is called \emph{biextendable}\index{biextendable set}
if  every $w\in F$ can be extended on the left
and on the right, that is such that $\ell_F(w)>0$ and $r_F(w)>0$.

A biextendable set is a \emph{tree set}\index{tree set} if
 for every $w\in F$, the graph $E_F(w)$
is a tree. A tree set is neutral.

More generally one also defines a \emph{connected}\index{connected set}
(resp. \emph{acyclic}\index{acyclic set}), as a biextendable set $F$
such that for every $w\in F$, the graph $E_F(w)$ is connected
(resp. acyclic). Thus a biextendable set is a tree set
if and only if it is both connected and acyclic.
\begin{example}
The Fibonacci set is a tree set. This follows from the fact that it is
a Sturmian set (see~\cite{BertheDeFeliceDolceLeroyPerrinReutenauerRindone2013a}) and that every Sturmian set is a tree set.
\end{example}
\begin{example}
The \emph{Tribonacci set}\index{Tribonacci set} is the set of factors of the fixed point
of the morphism $\varphi:a\mapsto ab,b\mapsto ac,c\mapsto a$.
It is a also a tree set (see~\cite{BertheDeFeliceDolceLeroyPerrinReutenauerRindone2013a}). The graph $E(\varepsilon)$ is represented in Figure~\ref{figureTribonacci}.
\begin{figure}[hbt]
\centering
\gasset{Nadjust=wh,AHnb=0}
\begin{picture}(30,20)
\node(al)(0,20){$a$}\node(bl)(0,10){$b$}\node(cl)(0,0){$c$}
\node(ar)(30,0){$a$}\node(br)(30,10){$b$}\node(cr)(30,20){$c$}

\drawedge(al,ar){}\drawedge(al,br){}\drawedge(al,cr){}
\drawedge(bl,ar){}\drawedge(cl,ar){}
\end{picture}
\caption{The extension graph of $\varepsilon$ in the Tribonacci set.}
\label{figureTribonacci}
\end{figure}
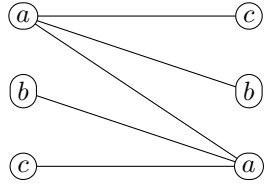
\end{example}
\begin{example}\label{exampleMorphismAC4}
Let $\varphi:a\mapsto ab,b\mapsto a^3b$ be as in Example~\ref{exampleMorphismAC}. Let $F$ be the set of factors of $\varphi^\omega(a)$.
 The graphs $E(a)$ and $E(aa)$ are
shown in Figure~\ref{figureExampleAC}.
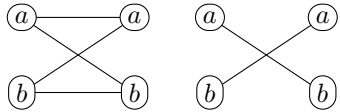
\begin{figure}[hbt]
\centering
\gasset{Nadjust=wh,AHnb=0}
\begin{picture}(50,20)
\put(0,0){
\begin{picture}(20,10)
\node(al)(0,10){$a$}\node(ar)(15,10){$a$}
\node(bl)(0,0){$b$}\node(br)(15,0){$b$}

\drawedge(al,ar){}\drawedge(al,br){}
\drawedge(bl,ar){}\drawedge(bl,br){}
\end{picture}
}
\put(25,0){
\begin{picture}(20,10)
\node(al)(0,10){$a$}\node(ar)(15,10){$a$}
\node(bl)(0,0){$b$}\node(br)(15,0){$b$}

\drawedge(al,br){}
\drawedge(bl,ar){}
\end{picture}
}
\end{picture}
\caption{The extension graphs $E(a)$ and $E(aa)$.}
\label{figureExampleAC}
\end{figure}
The first graph has a cycle of length $4$ and the second one has
two connected components.
Thus $F$ is not a tree set (it is not either an acyclic or a connected set,
as defined in~\cite{BertheDeFeliceDolceLeroyPerrinReutenauerRindone2013a}).
\end{example}
\begin{example}\label{exampleThueMorse2}
Let  $T$ be the Thue-Morse set (see example~\ref{exampleThueMorse}).
 The set $F$ is not a tree set
since $E_T(\varepsilon)$ is the complete bipartite graph $K_{2,2}$ on two sets with $2$ elements.
\end{example}

%%%%%%%%%%%%%%%%%%%%%%%%%%%%%%%
\section{Return words}\label{sectionReturn}
In this section, we introduce return words. We begin with the
classical notion of left and right return words in factorial sets.
We then develop a notion of limit return sets for pseudowords.
\subsection{Left and right return words}
Let $F$ be a factorial set.
A \emph{return word}\index{return word} to $x\in F$ is a nonempty word $w\in F$ 
such that $xw$ begins and ends by $x$
  but has no internal factor equal to $x$. We denote by $\RR_F(x)$ the set of return words to $x$.

For $x\in F$, we denote
\begin{displaymath}
\Gamma_F(x)=\{w\in F\mid xw\in F\cap A^*x\}.
\end{displaymath}
Thus $\RR_F(x)$ is the set of nonempty words in $\Gamma_F(x)$ without
any proper prefix in $\Gamma_F(x)$. Note that $\Gamma_F(x)$ is
a right unitary submonoid of $A^*$.

One also defines a \emph{left return word}\index{left return word}
\index{return word!left} to $x\in F$
as a nonempty word such that $wx$ begins and ends with $x$
but has no internal factor equal to $x$. We denote by $\RR'_F(x)$
the set of left return words to $x$. One has obviously
$\RR'_F(x)=x\RR_F(x)x^{-1}$.

\begin{example}\label{exampleReturnFibonacci}
Let  $F$ be the Fibonacci set. The sets
of right and left return words to $a$ are $\RR_F(a)=\{a,ba\}$
and $\RR'_F(a)=\{a,ab\}$. Similarly,
$\RR_F(b)=\{ab,aab\}$ and $\RR'_F(b)=\{ba,baa\}$.
\end{example}

\begin{example}\label{exampleReturnPeriodic}
Let $F$ be a periodic set. Let $w$ be a primitive word of length $n$ such that
$F=F(w^*)$. Then, for any word $x\in F$ of length at least
$n$, the set $\RR_F(w)$ is reduced to one word of length $n$.
\end{example}

The following is~\cite[Equation (4.2)]{BertheDeFeliceDolceLeroyPerrinReutenauerRindone2013m}.
\begin{proposition}\label{propositionReturnGenerators}
Let $F$ be a factorial set.
For any $x\in F$, one has $\Gamma_F(x)=\RR_F(x)^*\cap x^{-1}F$.
\end{proposition}
\begin{proof}
If a nonempty word $w$ is in $\Gamma_F(x)$ and is not in $\RR_F(x)$,
then $w=uv$ with $u\in \Gamma_F(x)$ and $v$ nonempty. Since
$\gamma_F(x)$ is right unitary, we have $v\in \Gamma_F(x)$,
whence the conclusion $w\in \RR_F(x)^*$ by induction on the length
of $w$. Moreover, one has $xw\in F$ and thus $w\in x^{-1}F$.

Conversely, assume that $w$ is a nonempty word
 in $\RR_F(x)^*\cap w^{-1}F$. Set 
$w=uv$ with $u\in \RR_F(x)$ and $v\in \RR_F(x)^*$. Then
$xw=xuv\in A^*xv\subset A^*x$ and $xw\in F$. Thus $w\in \Gamma_F(x)$.
\end{proof}
Note that, as a consequence, for $x,y\in F$ such that $xy\in F$, we have
\begin{equation}
\RR_F(xy)\subset \RR_F(y)^*.\label{eqReturnSuf}
\end{equation}
Indeed, if $w\in\RR_F(xy)$, then $xyw\in F\cap A^*xy$ implies
$yw\in F\cap A^y$ and thus the result follows since
$\Gamma_F(y)\subset \RR_F(y)^*$ by Proposition~\ref{propositionReturnGenerators}.

 The dual of Proposition~\ref{propositionReturnGenerators} 
and of Equation~\eqref{eqReturnSuf}
hold for left return words.

By a result of~\cite{BalkovaPelantovaSteiner2008}, in a uniformly recurrent
neutral set, one has
\begin{equation}
\Card(\RR_F(x))=\Card(A)\label{eqReturn}
\end{equation}
for every $x\in F$. 
\begin{example}\label{exampleMorphismAC3}
Let $\varphi:a\mapsto ab,b\mapsto a^3b$ be as in Example~\ref{exampleMorphismAC}. Let $F$ be the set of factors of $\varphi^\omega(a)$.
One has $\RR_F(a)=\{a,ba\}$ but
$\RR_F(aa)=\{a,babaa,babababaa\}$. Thus the number of return
words is not constant in a uniformly recurrent set which is not neutral.
\end{example}

There is an explicit form for the return words in a Sturmian set
(see~\cite[Theorem 4.4, Corollaries~4.1 and~4.5]{JustinVuillon2000}).
\begin{proposition}\label{propJustinVuillon}
  Let $s$ be a standard strict episturmian word over $A$, let
  $\Delta=a_0a_1\cdots$ be its directive word, and let $(u_n)_{n\ge 0}$ be its
  sequence of palindrome prefixes. 
  \begin{enumerate}
  \item[\upshape{(i)}] The  left return words to $u_n$ are the
    words $\psi_{a_0\cdots a_{n-1}}(a)$ for $a\in A$.
  \item[\upshape{(ii)}] For each factor $u$ of $s$, let $n$
be the minimal integer such that $u$ is a factor of $u_n$.
There there is a unique word
    $z$ such that $zu$ is a prefix of $u_n$ and the left return words to
    $u$ are the words $z^{-1}yz$, where $y$ ranges over the  left
    return words to $u_n$.
  \end{enumerate}
\end{proposition}

\begin{example}\label{exampleReturnFibo}
Let $\varphi$ be the Fibonacci morphism and
let $F$ be the Fibonacci set. We have for $n\ge 1$,
\begin{equation}
\RR'_F(\varphi^{2n}(aa))=\{\varphi^{2n}(a),\varphi^{2n+1}(a)\}.\label{eqReturnFibo}
\end{equation}
For example, $\varphi^2(aa)=abaaba$ and $\RR'_F(\varphi^2(aa))=\{aba,abaab\}$.
Note that Equation~\eqref{eqReturnFibo} does not hold for $n=0$.
Indeed, $\RR'_F(aa)=\{aab,aabab\}\ne\{a,ab\}$.

To show \eqref{eqReturnFibo}, we first observe that in the prefix order
for $n\ge 1$,
\begin{equation}
u_{2n}<\varphi^{2n}(aa)\le u_{2n+1}.\label{eqdouble}
\end{equation}
Indeed, one has by \eqref{eqMagique}
\begin{equation}
u_{2n+1}=\varphi^{2n}(a)u_{2n},
\label{equ_2n+1}
\end{equation}
and, since $n\ge 1$,
\begin{equation}
u_{2n}=\varphi^{2n-1}(a)u_{2n-1}=\varphi^{2n-1}(a)\varphi^{2n-2}(a)u_{2n-2}=\varphi^{2n}(a)u_{2n-2}.\label{equ_2n}
\end{equation}
Thus $u_{2n+1}=\varphi^{2n}(a)u_{2n}=\varphi^{2n}(aa)u_{2n-2}$. This
proves the second inequality in~\eqref{eqdouble}. Since
$|u_{2n}|\le2|u_{2n-1}|$ by \eqref{ineqlg}, and $|u_{2n-1}|<|\varphi^{2n}(a)|$
by \eqref{eqInegLg}, we obtain $|u_{2n}|<2|\varphi^{2n}(a)|$
and this proves the first inequality .

By \eqref{eqdouble}, the minimal integer $m$
such that $\varphi^{2n}(aa)$ is a factor of $u_m$ is $m={2n+1}$.
Thus, by Proposition~\ref{propJustinVuillon} (ii), one
has $\RR'_F(\varphi^{2n})(aa)=\RR'_F(u_{2n+1})$. On the other hand, 
by Proposition ~\ref{propJustinVuillon} (i), we have
$\RR'_F(u_{2n+1})=\{\psi_{(ab)^na}(a),\psi_{(ab)^na}(b)\}$. Since
$\psi_{(ab)^na}(a)=\varphi^{2n}(a)$ and $\psi_{(ab)^na}(b)=\varphi^{2n}(ab)
=\varphi^{n+1}(a)$, this proves \eqref{eqReturnFibo}.

\end{example}

The following is \cite[Theorem 4.5]{BertheDeFeliceDolceLeroyPerrinReutenauerRindone2013a}. It shows that in a tree set, a property much
stronger than Equation~\eqref{eqReturn} holds.
\begin{theorem}\label{theoremReturn}
Let $F$ be a uniformly recurrent tree set. For any $x\in F$, the
set $\RR_F(x)$ is a basis of $\F(A)$.
\end{theorem}
The proof uses Equation~\eqref{eqReturn} and the following
result~\cite[Theorem 4.7]{BertheDeFeliceDolceLeroyPerrinReutenauerRindone2013a}.
\begin{theorem}\label{theoremReturnConnected}
Let $F$ be a uniformly recurrent connected set. For any $w\in F$,
the set $\RR_F(w)$ generates the free group $\F(A)$.
\end{theorem}

\begin{example}
Let $F$ be the Tribonacci set on $A=\{a,b,c\}$. Then $\RR_F(a)=\{a,ba,ca\}$, which is
easily seen to be a basis of $\F(A)$.
\end{example}
%%%%%%%%%%%%%%%%%%%%%
\subsection{Limit return sets}\label{sectionLimitReturn}
Let $k\ge 1$ be an integer.
A  recurrent set $F$ is \emph{$k$-bounded}\index{$k$-bounded}
if  every $x\in F$ has at most $k$ return words. The set
$F$ is \emph{bounded}\index{bounded} if it is $k$-bounded for
some $k$. 

Thus, by Equation~\eqref{eqReturn}, a neutral set is $k$-bounded
with $k=\Card(A)$.

Clearly, any bounded set is uniformly recurrent.
There exist uniformly recurrent sets which are not bounded
 (see~\cite[Example 3.17]{DurandLeroyRichomme2013}).

Let $F$ be a $k$-bounded set. Consider an element $x$ which belongs to a group
in $J(F)$.
Let $r_n$ be a  sequence of finite prefixes
of $x$ strictly increasing for the prefix order. Similarly, let
$\ell_n$ be a sequence of finite suffixes of $x$ stricly increasing
for the suffix order. Since $x^2\in J(F)$, we have $\ell_nr_n\in F$. Let 
\begin{displaymath}
\RR_n=r_n\RR_F(\ell_nr_n)r_n^{-1}.
\end{displaymath}
Up to taking a subsequence, we may assume that the set
$\RR_n$ has a fixed number $\ell\le k$ of elements 
$r_{n,1},\ldots r_{n,\ell}$ and that
the sequence $(r_{n,1},\ldots,r_{n,\ell})$ is convergent in $\widehat{A^*}^\ell$. Its
limit $\RR$ is called a \emph{limit return set}\index{limit return set}
 to $x$. The sequence $(\ell_n,r_n,\RR_n)$ is called an \emph{approximating sequence}\index{approximating sequence}
for $\RR$.

We note that 
\begin{equation}
\RR_{n+1}\subset \RR_n^*.\label{eqRn}
\end{equation}
Indeed, set $r_{n+1}=r_ns_n$.
Since $\ell_{n+1}r_n$ is a prefix of $\ell_{n+1}r_{n+1}$, we have
by the dual of \eqref{eqReturnSuf} the inclusion
$\RR'_F(\ell_{n+1}r_{n+1})\subset \RR'_F(\ell_{n+1}r_n)^*$. Thus
\begin{eqnarray*}
\RR_F(\ell_{n+1}r_{n+1})&=&r_{n+1}^{-1}\ell_{n+1}^{-1}\RR'_F(\ell_{n+1}r_{n+1})\ell_{n+1}r_{n+1},\\
&\subset &r_{n+1}^{-1}\ell_{n+1}^{-1}\RR'_F(\ell_{n+1}r_{n})^*\ell_{n+1}r_{n+1}\\
&\subset &s_n^{-1}\RR_F(\ell_{n+1}r_{n})^*s_n.
\end{eqnarray*}
Since $\ell_nr_n$ is a suffix of $\ell_{n+1}r_n$, we have by~\eqref{eqReturnSuf}
the inclusion $\RR_F(\ell_{n+1}r_{n})\subset \RR_F(\ell_{n}r_n)^*$. Thus
we obtain
\begin{eqnarray*}
\RR_{n+1}&=&r_{n+1}\RR_F(\ell_{n+1}r_{n+1})r_{n+1}^{-1},\\
&\subset&r_{n+1}s_n^{-1}\RR_F(\ell_{n+1}r_{n})^*s_nr_{n+1}^{-1},\\
&\subset& r_n\RR_F(\ell_nr_n)^*r_n^{-1}=\RR_n^*.
\end{eqnarray*}
\begin{example}\label{exampleGroupFibonacci}
Let $\varphi$ be the Fibonacci morphism and $F$ be
 the Fibonacci set. We will show that 
$\{\varphi^\omega(a),\varphi^\omega(ba)\}$ is a limit return
set to $x=\varphi^\omega(a)$.

For this, consider the sequence $\ell_n=r_n=\varphi^{2n}(a)$. The sequence
is increasing both for the prefix and the suffix order and its
terms are both prefixes and suffixes of $x$. The set 
$\RR_n=r_n\RR_F(\ell_nr_n)r_n^{-1}=\ell_n^{-1}\RR'_F(\ell_nr_n)\ell_n$
 is by \eqref{eqReturnFibo}
\begin{eqnarray*}
\RR_n&=&\varphi^{2n}(a)^{-1}\{\varphi^{2n}(aa),\varphi^{2n+1}(a)\}\varphi^{2n}(a)
=\{\varphi^{2n}(a),\varphi^{2n}(b)\varphi^{2n}(a)\}\\
&=&\{\varphi^{2n}(a),\varphi^{2n}(ba)\}.
\end{eqnarray*} 
The sequence $(\varphi^{n!}(a),\varphi^{n!}(ba)\}$ converges to
$\{\varphi^\omega(a),\varphi^\omega(ba)\}$, which proves the claim.
\end{example}

The following example shows that in degenerated cases, a limit
return set may contain finite words.
\begin{example}\label{exampleReturnPeriodic2}
Let $F$ be a periodic set of period $n$. By Example~\ref{exampleReturnPeriodic}
a return set to any word of length larger than $n$ is formed
of one word of length $n$. Thus any return set to a pseudoword
$x\in J(F)$ is formed of one word of length $n$.
\end{example}
%%%%%%%%%%%%%%%%%%%%%%%%%
\section{Sch\"utzenberger groups} \label{sectionSchutzenbergerGroups}

Let $M$ be a topological monoid.
For an element $x\in M$, we denote by $H(x)$
the $\HH$-class of $x$. 

Let $H$ be an $\HH$-class of $M$. Set $T(H)=\{x\in M\mid Hx=H\}$.
Each $x\in T(H)$ defines a map $\rho_x:H\rightarrow H$
defined by $\rho_x(h)=hx$.
The set of the translations $\rho_x$ for $x\in T(H)$
is a topological group acting by permutations on $H$,
denoted $\Gamma(H)$. The groups corresponding to different
$\HH$-classes contained in the same $\JJ$-class $J$
are continuously isomorphic and the
equivalence 
called the \emph{Sch\"utzenberger group}\index{Sch\"utzenberger group}
\index{group!Sch\"utzenberger} of $J$, denoted $G(J)$.

If $J$ is a regular $\JJ$-class, any $\HH$-class of $J$ which is a group
is isomorphic to $G(J)$. Indeed, $H\subset T(H)$
and the restriction to $H$ of the mapping 
$\rho:x\in T(H)\rightarrow \rho_x\in \Gamma(H)$ is an isomorphism
(see~\cite{Lallement1979} for a more detailed presentation).

The following is \cite[Proposition 5.2]{AlmeidaCosta2013}.
\begin{theorem}\label{theoremPresentationTree}
Let $F$ be a non-periodic bounded set. Let  $x\in \JJ(F)$
be such that $H(x)$ is a group and let $\RR$ be a limit return
set to $x$. Then $H(x)$ is the closure of the semigroup
generated by $\RR$.
\end{theorem}
Note that Theorem~\ref{theoremPresentationTree} does not hold
without the hypothesis that $F$ is non-periodic (see Example~\ref{exampleReturnPeriodic2}).

This fundamental result is the key to understand the role
played by the group generated by return words. Actually,
let $(\ell_n,r_n,\RR_n)$ be an approximating sequence for $\RR$.
Then, since $\RR_{n+1}\subset \RR_n^*$ by \eqref{eqRn},
we have $\RR^*=\cap_{n\ge 0}\overline{\RR_n^*}$. 
Thus the group $H(x)$ is the intersection of the submonoids $\overline{\RR_n^*}$
and each of them is the closure of the submonoid generated by $\RR_n$.

\begin{example}
Let  $\varphi$ be the Fibonacci morphism and let
$F$ be  the Fibonacci set. The $\HH$-class of the
pseudoword $x=\varphi^\omega(a)$
is a group. Indeed, $H(x)$ contains the
idempotent $\varphi^\omega(a^\omega)$.  The group $H(x)$
is the closure of the semigroup generated by $x$ and $y=\varphi^\omega(ba)$, that
is, isomorphic to $\widehat{F(A)}$.
\end{example}

The following statement is a 
generalization of~\cite[Theorem 6.5]{AlmeidaCosta2015}.
We denote by $G(F)$ the Sch\"utzenberger group of $J(F)$.
\begin{theorem}\label{theoremGene}
Let $F$ be a non-periodic bounded set and let $f:\widehat{A^*}\rightarrow G$
be a continuous morphism from $\widehat{A^*}$ onto a profinite group $G$.
The following conditions are equivalent.
\begin{enumerate}
\item[\rm (i)] The restriction of $f$ to $G(F)$ is surjective.
\item[\rm (ii)] For every $w\in F$, the submonoid $f(\RR_F(w)^*)$
is dense in $G$.
\end{enumerate}
\end{theorem}
\begin{proof}
 Let $x\in J(F)$ be such that $H(x)$ is a group. Let
$(\ell_n,r_n)$ be a sequence of pairs
of suffixes and prefixes of $x$ of strictly increasing length.
Taking a subsequence, we may assume
 that $\RR_n=r_n\RR_F(\ell_nr_n)r_n^{-1}$
converges to the limit return set $\RR$. 
Since, by~\eqref{eqRn} we have $\RR_{n+1}\subset \RR_n^*$, the 
semigroup generated by $\RR$ is $\cap_{n\ge 0}\overline{\RR_n^*}$.

(i) implies (ii). Assume by contradiction that $f(\RR_F(w)^*)$ is not
dense in $G$.
Since $w$ is a factor of $x$, we may assume that  $r_0$
ends with $w$.  Then
 $\RR_F(\ell_0r_0)\subset \RR_F(w)^*$. This implies that
$f(\RR_F(\ell_0r_0)^*)$ is not dense in $G$.  Since $f(\RR_0^*)$ is conjugate
to $f(\RR_F(\ell_0r_0)^*)$, the same holds for $f(\RR_0^*)$.
 Thus
$f(\RR^*)$ is not dense in $G$. But by Theorem~\ref{theoremPresentationTree}, $H(x)$
is the closure of the semigroup generated by $\RR$.
We conclude that $f(H(x))$ is not dense in $G$.

(ii) implies (i). Since $H(x)= \RR^*=\cap_{n\ge 0}\overline{\RR_n^*}$,
we have $f(H(x))=\cap_{n\ge 0}\overline{f(\RR_n^*)}$. But $\RR_n$
is conjugate to $\RR(\ell_nr_n)$ and thus by (ii), each $f(\RR_n^*)$
is dense in $G$. Thus $f(H(x))=G$.
\end{proof}

\begin{corollary}\label{corollary1}
Let $F$ be a non-periodic bounded  set on the alphabet $A$. The following conditions are equivalent.
\begin{enumerate}
\item[\rm(i)] The restriction to any maximal
subgroup of $J(F)$ of the natural projection $p_G:\widehat{A^*}\rightarrow \widehat{\F(A)}$
is surjective.
\item[(ii)] For each $w\in F$ the set $\RR_F(w)$ generates the free
group $\F(A)$.
\end{enumerate}
\end{corollary}
\begin{proof}
We apply Theorem~\ref{theoremGene} with $f$ being the identity.
Let $x\in J(F)$ be such that $H(x)$ is a group.

(i) implies (ii). Let $w\in F$. There is a maximal subgroup
of $J(F)$ contained in the topological closure
of $R_F(w)^*$ in the free profinite monoid
(indeed,  $RR_F(l_nr_n)^*$  is a subset of $\RR_F(w)^*$, for some suitable
sequence of words $l_nr_n$ as defined some pages before, for infinitely
many $n$).

It follows that the topological closure of $R_F(w)^*$ in the free
profinite group generated by $A$ is the whole free profinite group. This
implies that $\RR_F(w)$ generates $\F(A)$.

(ii) implies (i). By Theorem~\ref{theoremGene}, the restriction to $H(x)$
of the projection
$p_G:\widehat{A^*}\rightarrow \widehat{\F(A)}$ is surjective. 
\end{proof}

\subsection{Groups of tree sets}
We now consider uniformly recurrent tree sets. Note that a uniformly
recurrent tree set $F$ is non-periodic. Indeed, if $F$ is the set
of factors of $w^*$ with $w$ primitive, then $\RR_F(w)=\{w\}$ since
$w$ does not overlap nontrivially $w^2$. Thus $F$ is not
a tree set by Theorem~\ref{theoremReturn}.
\begin{theorem}\label{corollary2}
Let $F$ be a uniformly recurrent tree set.  Then the following assertions hold.
\begin{enumerate}
\item The group $G(F)$
is the free profinite group on $A$. More precisely, the restriction to any maximal
subgroup  of $J(F)$ of the natural projection $p_G:\widehat{A^*}\rightarrow \widehat{\F(A)}$
is an isomorphism.
\item Let  $H$ be a subgroup
of finite index $n$ in $\F(A)$. For any maximal group $G$ in $J(F)$,
$G\cap p_G^{-1}(\bar{H})$ is a subgroup of index $n$ of $G$.
\end{enumerate}
\end{theorem}
\begin{proof}
1.
This results directly from Corollary~\ref{corollary1} since by
Theorem~\ref{theoremReturn}, $\RR_F(w)$ is a basis
of $\F(A)$ for every $w\in F$ when $F$ is a uniformly recurrent tree set.
Since $H(x)$ is the closure
of a semigroup generated by  $\Card(A)$ 
elements, there is a continuous morphism $\psi$ from $\widehat{\F(A)}$
onto $H(x)$. Thus $p_G \circ \psi$ is continuous surjective morphism
from $\widehat{\F(A)}$ onto itself. By Proposition~\ref{propositionHopfian},
it implies that it is an isomorphism. This proves the first
assertion.

2. This results from Corollary~\ref{corollary1} since the restriction 
$\alpha$ of
$p_G$ to $G$ is an isomorphism 
from $G$ onto $\widehat{\F(A)}$ and $G\cap p_G^{-1}(\bar{H})=\alpha^{-1}(\bar{H})$.
\end{proof}
\begin{example}
Let $F$ be the Fibonacci set. We have seen that $G(F)$ is
the free profinite group on $A$ (Example~\ref{exampleGroupFibonacci}).
\end{example}

%%%%%%%%%%%%%%%%%%%%%%%%%%%%%%%%%%%%%%
\subsection{Groups of fixed points of morphisms}
Let $\varphi:A^*\rightarrow A^*$ be a primitive substitution and let $F(\varphi)$
be the set of factors of a fixed point of $\varphi$. We denote by $J(\varphi)$
the closure of $F(\varphi)$ and by $G(\varphi)$ the Sch\"utzenberger group
of $J(\varphi)$.

A \emph{connexion}\index{connexion}\index{morphism!connexion} for $\varphi$ is 
a word $ba$ with $b,a\in A$ such that $ba\in F(\varphi)$, the first
letter of $\varphi^\omega(a)$ is $a$ and the last letter of $\varphi^\omega(b)$
is $b$. Every primitive substitution has a connexion~\cite[Lemma 4.1]{Almeida2005b}. A \emph{connective power} of $\varphi$ is a finite power $\tilde{\varphi}$ of $\varphi$ such that the first letter of 
$\tilde{\varphi}(a)$ is $a$, the last letter of $\tilde{\varphi}(b)$ is $b$.
We denote $X_\varphi(a,b)=aR_F(ba)a^{-1}$. The set $X_\varphi(a,b)$ is a code.

\begin{example}\label{exampleThueMorse3}
Let $\tau:a\mapsto ab,b\mapsto ba$ be the Thue-Morse morphism.
The word $aa$ is a connection for $\tau$ and $\tilde{\tau}=\tau^2$ is a connecting power of $\tau$.
The set $X=X_\tau(a,a)$ has four elements $x=abba$, $y=ababba$, $z=abbaba$ and $t=ababbaba$.
\end{example}
The following is~\cite[Theorem 5.6]{AlmeidaCosta2013}.
\begin{theorem}\label{theorem5.6}
Let $\varphi$ be a non periodic primitive substitution. Consider a connexion $ba$ for $\varphi$ and a
connective power $\tilde{\varphi}$. The intersection
$H_{ba}$ of the $\RR$-class of  $\varphi^\omega(a)$ with the
$\LL$-class of $\varphi^\omega(b)$
is a group and  $H_{ba}=\tilde{\varphi}(H_{ba})=\varphi^\omega(H_{ba})$.
\end{theorem}

The proof uses the notion of recognizablity of a substitution. We give the definition
in the following form (see~\cite{KloudaStarosta2014} for the equivalence with equivalent
forms). Given a morphism $\varphi:A^*\rightarrow A^*$, a pair $(q,r)$ of words in $A^*$
is \emph{synchronizing}\index{synchronizing!pair} if for any $p,s,t\in A^*$ such that $\varphi(t)=pqrs$,
one has $t=uv$ with $\varphi(u)=pq$ and $\varphi(v)=rs$ (see Figure~\ref{figureRecognizability}).
\begin{figure}[hbt]
\centering\gasset{Nadjust=wh,AHnb=0}
\begin{picture}(60,15)(0,-2)
\node(ul)(15,10){}\node(ur)(30,10){}\node(vr)(45,10){}
\node(pl)(0,0){}\node(ql)(10,0){}\node(rl)(30,0){}\node(sl)(50,0){}\node(sr)(60,0){}

\drawedge(ul,ur){$u$}\drawedge(ur,vr){$v$}
\drawedge[ELside=r](pl,ql){$p$}\drawedge[ELside=r](ql,rl){$q$}\drawedge[ELside=r](rl,sl){$r$}\drawedge[ELside=r](sl,sr){$s$}
\gasset{AHnb=1}
\drawedge[ELside=r](ul,pl){$\varphi$}\drawedge(ur,rl){}\drawedge(vr,sr){$\varphi$}
\end{picture}
\caption{A synchronizing pair.}\label{figureRecognizability}
\end{figure}
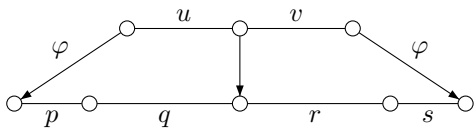
Let $F$ be the set of factors of a fixed point of $\varphi$.
The morphism $\varphi$ is \emph{recognizable}\index{recognizable!morphism} if there is an integer $n\ge 1$ such that
for any $x,y\in F\cap A^n$ such that $xy\in F$, the pair $(\varphi(x),\varphi(y))$
is synchronizing. By a result of Moss\'e~\cite{Mosse1992}, any non-periodic primitive substitution
is recognizable (see~\cite{KloudaStarosta2014} for a new version of the proof).

The following is the main result if~\cite{AlmeidaCosta2013} (Theorem 6.2).
\begin{theorem}\label{theoremPresentation}
Let $\varphi$ be a non-periodic primitive substitution over the alphabet $A$. Let $ba$ be a connexion
of $\varphi$ and let $X_\varphi=X_\varphi(a,b)$. Then $G(\varphi)$ admits the presentation
\begin{displaymath}
\langle X\mid \tilde{\varphi}_{X,G}^\omega(x)=x, x\in X\rangle_ G.
\end{displaymath}
\end{theorem}

\begin{example}\label{exampleThueMorse4}
Let $\tau:a\mapsto ab,b\mapsto ba$ be the Thue-Morse morphism. We have seen in Example~\ref{exampleThueMorse3} that
 the word $aa$ is a connection for $\tau$ and $\tilde{\tau}=\tau^2$ is a connecting power of $\tau$.
The set $X=X_\tau(a,a)$ has four elements $x=abba$, $y=ababba$, $z=abbaba$ and $t=ababbaba$.
By Theorem~\ref{theoremPresentation}, the group $G(\tau)$ is the group generated by
$X$ with the relations $\tau_{X,G}^\omega(u)=u$ for $u\in X$. Actually, since
$\tau^\omega(y)\tau^{\omega-1}(x)\tau^\omega(z)=\tau^\omega(t)$, the relation $xy^{-1}z=t$ is a consequence
of the relations above and thus $G(\varphi)$ is generated by $x,y,z$.
\end{example}

Let $f:A^*\rightarrow G$ be a morphism from $A^*$ into a finite group $G$
and let $\varphi:A^*\rightarrow A^*$ be a morphism. We denote by
$\varphi_G$ the map from $G^A$ into itself defined as follows.
Consider $h\in G^A$. We may naturally extend $h$ to a map from $A^*$
into $G$.
For $a\in A$, we define the image of $a$ by $\varphi_G(h)$ as
$\varphi_G(h)(a)=h(\varphi(a))$. We say that $\varphi$ has \emph{finite
$f$-order} if there is an integer $n\ge 1$ such that $\varphi^n_G(f)=f$.
The least such integer is called the $f$-\emph{order}\index{substitution!order of} of $\varphi$.

Any substitution $\varphi$ which is invertible in $\F(A)$ is of finite
$h$-order for any morphism $h$ into a finite group. Indeed, since $G$
is finite, there are integers $n,m$ with $n<m$ such that
$\varphi_G^{n+m}=\varphi^n_G$. Since $\varphi$ is invertible,  $\varphi_G^m$
is the identity.

\begin{example}
Let $\varphi:a\mapsto ab,b\mapsto a$ be the Fibonacci substitution and
let $h:A^*\rightarrow \Z/2\Z$ be the parity of the length, that is the morphism into the additive goup
of integers modulo $2$ sending each letter to $1$. Then 
$\varphi$ is of $h$-order $3$.
\end{example}

The following is a consequence of Theorem~\ref{theoremPresentation},
using~\cite[Proposition 3.2]{AlmeidaCosta2013}.
\begin{corollary}\label{corollaryPresentation}
Let $\varphi$ be a non-periodic primitive substitution over $A$ and let $h:A^*\rightarrow G$
be a morphism onto a finite group.
The restriction of $\hat{h}:\widehat{A^*}\rightarrow G$ to any maximal
subgroup of $J(\varphi)$ is surjective if and only if $\varphi$
has finite $h$-order.
\end{corollary}

\begin{example}
Let $\varphi$ be as in Example~\ref{exampleMorphismAC4}, let $G=\Z/2\Z$ and let $h:A^*\rightarrow G$ be the
parity of the length. Then $\varphi_G(h)=(0,0)$ and $\varphi_G(0,0)=(0,0)$.
Thus $\varphi$ does not have finite $h$-order. Actually, 
any pseudoword in $J(\varphi)$ which is in the image of $\hat{\varphi}$
has even length and thus is mapped by $h$ to $0$. 
Thus, by Theorem~\ref{theorem5.6}, there is a maximal group $G$ in $J(F)$
which contains only pseudowords of even length and therefore
$\hat{h}(G)=\{0\}$, showing that the restriction of $\hat{h}$ to 
$G$
 is not surjective.
\end{example}
The following example is from~\cite[Section 7.2]{AlmeidaCosta2013}.
\begin{example}\label{exampleMorphismAC5}
Let $\varphi:a\mapsto ab,b\mapsto a^3b$ be as in the previous example and let
$h:A^*\rightarrow A_5$ be the morphism from $A^*$ onto the alternating group $A_5$
defined by $h:a\mapsto (123),b\mapsto (345)$. One may verify that $\varphi$
has $h$-order $12$. Thus $A_5$ is a quotient of $G(\varphi)$. It is not
known whether any finite group is a quotient of $G(\varphi)$.
\end{example}
\paragraph{Proper substitutions}
A substitution $\varphi$ over $A$ is \emph{proper}\index{substitution!proper} if there are letters
$a,b\in A$ such that for every $d\in A$, $\varphi(d)$ starts with $a$
and ends with $b$. Theorem~\ref{theoremPresentation} takes
a simpler form for proper substitutions.
The following is~\cite[Theorem 6.4]{AlmeidaCosta2013}.
\begin{theorem}\label{theoremPresentationProper}
Let $\varphi$ be  a non-periodic proper primitive substitution over
a finite alphabet $A$. Then $G(\varphi)$ admits the presentation
\begin{displaymath}
\langle A\mid \varphi_G^\omega(a)=a, a\in A\rangle_ G.
\end{displaymath}
\end{theorem}
The proof uses Proposition~\ref{lemma2.2} applied with the  diagram
of Figure~\ref{commutativeDiagram}.
\begin{figure}[hbt]
\centering
\gasset{Nframe=n,Nadjust=wh}
\begin{picture}(15,15)
\node(Sl)(0,0){$H$}\node(Sr)(15,0){$H$}
\node(A^*l)(0,15){$\widehat{A^*}$}\node(A^*r)(15,15){$\widehat{A^*}$}

\drawedge(Sl,Sr){$\varphi$}\drawedge(A^*l,A^*r){$\varphi^{\omega+1}$}
\drawedge(A^*l,Sl){$\varphi^\omega$}\drawedge(A^*r,Sr){$\varphi^\omega$}
\end{picture}
\caption{A commutative diagram}\label{commutativeDiagram}
\end{figure}
\begin{example}\label{exampleMorphismAC6}
Let $A=\{a,b\}$ and let $\varphi:a\mapsto ab,b\mapsto a^3b$. The morphism $\varphi$
is proper. Thus, by Theorem~\ref{theoremPresentationProper}, the Sch\"utzenberger group of $J(\varphi)$ has the presentation
$\langle a,b\mid \varphi_G^\omega(a)=a,\varphi_G^\omega(b)=b\rangle$.
Since the image of $\F(A)$ by $\varphi$  is included in the subgroup
generated by words of length $2$, the relations
$\varphi_G^\omega(a)=a$ and $\varphi_G^\omega(b)=b$
are nontrivial and thus $G(F)$ is not a free profinite
group of rank two (it is actually not a free
profinite group, see~\cite[Example 7.2]{Almeida2005b}).
\end{example}

%%%%%%%%%%%%%%%%%%%%%%%%%%%%%%%%%
\subsection{Groups of bifix codes}

For an automaton $\A$, we denote by $\varphi_\A$ the natural morphism
from $A^*$ onto the transition monoid of $\A$.

Let $F$ be a  recurrent set. For any finite automaton $\A=(Q,i,T)$, we denote by $\rank_\A(F)$
the minimum of the ranks of the maps $\varphi_\A(w)$ for $w\in F$.
By~\cite[Proposition 3.2]{Perrin2015}, the set of elements 
of $\varphi_\A(F)$ of rank $\rank_\A(F)$
is included in a regular $\JJ$-class, called the $F$-\emph{minimal} $\JJ$-class
of the monoid $\varphi_\A(A^*)$ and denoted $J_\A(F)$. The structure group of this $\JJ$-class is
denoted $G_\A(F)$.

Let $X\subset A^+$ be a code. A \emph{parse}\index{parse} of
 a word $w\in A^*$ with respect to $X$ is a triple $(p,x,q)$ with $w=pxq$ such that
$p$ has no suffix in $X$, $x\in X^*$ and $q$ has no prefix in $X$.

A parse of a profinite word $w\in \widehat{A^*}$ with respect to $X$
is a triple $(p,x,q)$
with $w=pxq$ such that
$p$ has no suffix in $X$, $x\in \widehat{X^*}$ and $q$ has no prefix in $X$.

The number of parses of $w\in \widehat{A^*}$ with respect
to a finite maximal prefix code $X$ is equal to the number
of its prefixes which have no suffix in $X$. Indeed the
map $(p,x,q)\mapsto p$ assigning to each parse its first component
 is bijective by Proposition~\ref{propositionUnitary2}.

Let $F$ be a recurrent set. A bifix code\index{bifix code}\index{code!bifix}
 $X\subset F$ is $F$-\emph{maximal}
if it is not properly contained in any bifix code $Y\subset F$. 

A bifix code $X\subset F$ is $F$-\emph{thin}\index{thin} if there is a word $w\in F$ which is
not a factor of $X$. When $F$ is uniformly recurrent, a set $X\subset F$
is $F$-thin if and only if it is finite.

Let $F$ be a recurrent set. The $F$-\emph{degree} of a bifix code $X$, denoted
$d_X(F)$, is the maximal number of parses
of a word in $F$. A bifix code $X$ is $F$-maximal and $F$-thin if and
only if its $F$-degree is finite. In this case a word of $F$ has $d_X(F)$
parses if and only if it is not an internal factor of a word of $X$
(see~\cite[Theorem 4.2.8]{BerstelDeFelicePerrinReutenauerRindone2012}).
\begin{proposition}
Let $F$ be a uniformly
recurrent set and let $X$ be a finite $F$-maximal bifix code.
The number of parses of any element of $\bar{F}$ is equal to $d_X(F)$.
\end{proposition}
\begin{proof}
Let $w\in\bar{F}$ and let $(u_n)$ be a sequence of elements of $F$
converging to $w$. Since each long enough $u_n$ has  $d_X(F)$ parses,
we may assume that all $u_n$ have  $d_X(F)$ of parses.
We may then number the parses of $u_n$ as $(p_{n,i},x_{n,i},q_{n,i})$
in such a way that for fixed $i$, each sequence converges
to $(p_i,x_i,q_i)$. Since $X$ is finite, the sequences $(p_{n,i})$
and $q_{n,i}$ are ultimately constant and the sequence
$(x_{n,i})$ converges to some $x_i\in \widehat{A^*}$.
Thus $w$ has $d_X(F)$ parses. There cannot exist more than
$d_X(F)$ parses of a word in $\bar{F}$ since the
number of parses is equal to the number of suffixes
which are prefixes of $X$.
\end{proof}
Let $X$ be an $F$-thin and $F$-maximal bifix code. The $F$-degree
of $X$ is equal to the $F$-minimal rank of the minimal automaton $\A$ of $X^*$.
We denote by $\varphi_X$ the morphism $\varphi_\A$, by $J_X(F)$ the $\JJ$-class $J_\A(F)$ and by $G_X(F)$
the group $G_\A(F)$, called
the $F$-\emph{group} of $X$.
It is a permutation group of degree $d_X(F)$.

We recall that for any uniformly recurrent tree set $F$
a finite bifix code $X\subset F$ is $F$-maximal of $F$-degree $d$ if
and only if it is a basis of a subgroup of index $d$
(Finite Index Basis Theorem, see~\cite[Theorem 4.4]{BertheDeFeliceDolceLeroyPerrinReutenauerRindone2015}).

We prove the following result. It has the interesting feature that
the hypothesis made on  profinite objects has a consequence
on finite words.
\begin{theorem}\label{newTheorem}
Let $F$ be a uniformly recurrent set, let $Z$ be a goup code of degree $d$
and let $X=Z\cap F$.
Let $h:A^*\rightarrow G$ be the morphism from $A^*$ 
onto the syntactic monoid of $Z^*$. The restriction of $\hat{h}$
to a maximal subgroup of $J(F)$ is surjective if and only if the following properties hold.
\begin{enumerate}
\item[\rm (i)] $X$ is an $F$-maximal bifix code of $F$-degree $d$.
\item[\rm (ii)] $G_X(F)$ is isomorphic to $G$.
\item[\rm (iii)] The morphism $\hat{\varphi_X}$ maps 
each $\HH$-class of $J(F)$ which is a group onto $G_X(F)$.
\end{enumerate}
\end{theorem}
\begin{proof}
Set $\varphi=\varphi_X$ and $M=\varphi(A^*)$. Let $(Q,i,i)$
be the minimal automaton of $Z^*$. Thus $G$ is a transitive
permutation group on the set $Q$ which has $d$ elements
and $h(Z^*)$ is the stabilizer of $i\in Q$.

By \cite[Theorem 4.2.11]{BerstelDeFelicePerrinReutenauerRindone2012}, 
since $F$ is recurrent,
the set $X$ is an $F$-thin $F$-maximal bifix code of $F$-degree at most $d$.
Since $F$ is uniformly recurrent, $X$ is finite.

Let $x\in J(F)$ be such that $H=H(x)$ is a group such that the restriction
of
$\hat{h}$ to $H(x)$ is surjective.

Since $\hat{h}$ maps $H(x)$ onto $G$, the pseudoword $x$
has $d$ parses with respect to $Z$ and thus with respect to $X$.
Indeed, for any $p\in Q$, $\hat{h}(x)$ sends $p$ on some $q\in Q$.
Then $x$ has the
interpretation $(u,v,w)$  with $ph(u)=i$, $i\hat{h}(v)=i$ and $ih(w)=q$.
This implies that $d_X(F)=d$ and proves (i).

It is clear that $\hat{\varphi}(J(F))$ is contained in $J_X(F)$
since $J_X(F)$ contains the image by $\varphi$ of every long enough
word of $F$. Thus $\hat{\varphi}(x)$ is in $J_X(F)\cap\varphi(F)$ and its
$\HH$-class $K$ is a group.

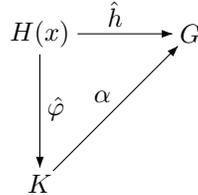
\begin{figure}[hbt]
\gasset{Nadjust=wh,Nframe=n}
\centering

\begin{picture}(20,20)
\node(A)(0,20){$H(x)$}\node(G)(20,20){$G$}
\node(M)(0,0){$K$}

\drawedge(A,G){$\hat{h}$}\drawedge(A,M){$\hat{\varphi}$}
\drawedge(M,G){$\alpha$}
\end{picture}

\caption{The reduction onto $G$.}\label{figureDiagram}
\end{figure}
Let $w$ be a word in $F$ and not a factor of $X$
such that $\varphi(w)\in K$. Let $P$ be the set
of suffixes of $w$ which are proper prefixes of $X$.
Since  $\Card(P)=d_X(F)$,
 $K$ is a permutation group on the set $\{i\cdot p\mid p\in P\}$ which is identified
by an isomorphism $\alpha$
with  a subgroup of $G$.

 If
the map $\hat{h}$ is surjective from $H(x)$ onto $G$, the commutativity
of the diagram in Figure~\ref{figureDiagram} forces $\alpha$
to be surjective. Moreover $\hat{\varphi}$ maps $H(x)$ onto $K$. The converse
is also true.
\end{proof}

In the case where $F$ is a tree set, the hypothesis of
Theorem~\ref{newTheorem} is satisfied by assertion 1
in Theorem~\ref{corollary2}. The conclusion  of
Theorem~\ref{newTheorem} is
implied by assertion 2.
\begin{example}\label{exampleDegree2}
Let $A=\{a,b\}$ and let $Z=A^2$. Let $F$ be the Fibonacci set.
Then $X=\{aa,ab,ba\}$. The group $G_X(F)$ is the cyclic group of
order $2$, in agreement with the fact
that $X$ generates  the kernel of the morphism from $\F(A)$ onto
$\Z/2\Z$ sending $a,b$ to $1$.
The minimal automaton of $X^*$ is shown in Figure~\ref{figureAutomaton}
on the left
and the $0$ minimal ideal of its transition monoid $M$ is represented
on the right.
\begin{figure}[hbt]
\centering
\gasset{Nadjust=wh}
\begin{picture}(80,20)
\put(0,0){
\begin{picture}(20,20)
\node[Nmarks=if,fangle=180](1)(0,10){$1$}\node(2)(20,20){$2$}\node(3)(20,0){$3$}
\drawedge[curvedepth=3](1,2){$a,b$}\drawedge[curvedepth=3,ELside=r](2,1){$a$}
\drawedge[curvedepth=3](1,3){$b$}\drawedge[curvedepth=3](3,1){$a$}
\end{picture}
}
\put(40,10){
$
\def\rb{\hspace{2pt}\raisebox{0.8ex}{*}}\def\vh{\vphantom{\biggl(}}
    \begin{array}%
    {r|@{}l@{}c|@{}l@{}c|}%
    \multicolumn{1}{r}{}&\multicolumn{2}{c}{1,2}&\multicolumn{2}{c}{1,3}\\
    \cline{2-5}
    1/2,3& \vh\rb &a &\vh \rb &ab \\
    \cline{2-5}
    1/2&\vh\rb&ba &\vh&b\\
    \cline{2-5}
    \end{array}
$
}
\end{picture}
\caption{The minimal automaton of $X^*$ and the $F$-minimal $\DD$-class.}
\label{figureAutomaton}
\end{figure}
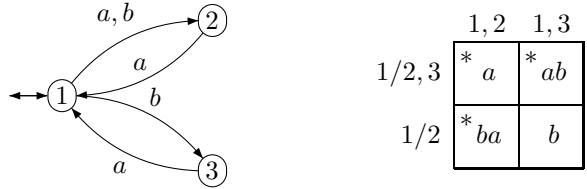
\end{example}
Note that, in the above example, the $F$-minimal $\DD$-class $D$ of $M$
is the image of the $\JJ$-class $J(F)$.
The following example shows that this may be
true although the image of $J(F)$ in the monoid $M$ is strictly 
included in $D$.
\begin{example}\label{exampleMorphismAC7}
Let $F$ be the set of factors of the fixed point of the morphism
$\varphi:a\mapsto ab,b\mapsto a^3b$ (as in Example~\ref{exampleMorphismAC}).
The set $F\cap A^2$ is the same as in Example~\ref{exampleDegree2}
and the $F$-minimal $\DD$-class is also the same. However,
$J(F)$ contains maximal groups
 formed of words of even length and thus its
image in $M$ is aperiodic, that is has trivial subgroups.
\end{example}

We now deduce from Corollary~\ref{corollaryPresentation} the following statement which gives
information on the groups $G_X(F)$ when $X$ is an $F$-maximal bifix
code in a set $F$ which is not a tree set. It would
be interesting to have a direct proof of this statement
which does not use profinite semigroups.

\begin{theorem}\label{theoremGroupCode}
Let $\varphi$ be a primitive non-periodic substitution over the alphabet $A$
and let $F$ be the set of factors of a fixed point of $\varphi$.  Let
$Z$ be a group code of degree $d$ on $A$ and let $h$ be the morphism from $A^*$ onto
the syntactic monoid of $Z^*$. Set $X=Z\cap F$. If $\varphi$ has finite $h$-order, then
$X$ is an $F$-maximal bifix code of $F$-degree $d$ and $G_X(F)$ is isomorphic to $G$.
\end{theorem}
\begin{proof}
By Corollary~\ref{corollaryPresentation}, the hypothesis of
Theorem~\ref{newTheorem} is satisfied and thus the conclusion using
conditions (i) and (ii).
\end{proof}

\begin{example}\label{exampleMorphismAC8}
Let $F$ and $\varphi$ be as in Example~\ref{exampleMorphismAC2}.
We consider, as in~\cite{AlmeidaCosta2013}, the morphim $h:A^*\rightarrow A_5$ from $A^*$ onto
the alternating group of degree $5$ defined by
$h:a\mapsto (123),b\mapsto (345)$. We have seen in Example~\ref{exampleMorphismAC5}
that $\varphi$ has $h$-order $12$ and thus, by Corollary~\ref{corollaryPresentation},
$\hat{h}$ induces a surjective map from any maximal subgroup of $J(\varphi)$
onto $A_5$. 

Let $Z$ be the bifix code
generating the submonoid stabilizing $1$ and let $X=Z\cap F$.
The $F$-maximal bifix code $X$ has $8$ elements. It is represented in Figure~\ref{figureBifixDegree5}
with the states of the minimal automaton indicated on its prefixes.
\begin{figure}[hbt]
\centering\gasset{Nadjust=wh,AHnb=0}
\begin{picture}(100,40)(0,-10)
\node(1)(0,5){$1$}
\node(a)(10,12){$2$}\node[Nmr=0](b)(10,0){$1$}
\node(aa)(20,17){$3$}\node(ab)(20,5){$4$}
\node[Nmr=0](aaa)(30,22){$1$}\node(aab)(30,14){$5$}\node(aba)(30,5){$6$}
\node(aaba)(40,14){$7$}\node[Nmr=0](abaa)(40,8){$1$}\node(abab)(40,0){$8$}
\node(aabab)(50,14){$9$}\node(ababa)(50,0){$10$}
\node(aababa)(60,14){$11$}\node(ababaa)(60,5){$12$}\node(ababab)(60,-5){$9$}
\node(aababaa)(70,19){$14$}\node(aababab)(70,10){$15$}
\node(ababaaa)(70,5){$16$}\node(abababa)(70,-5){$11$}
\node(aababaaa)(80,19){$17$}\node[Nmr=0](aabababa)(80,10){$1$}
\node(ababaaab)(80,5){$14$}\node(abababaa)(80,0){$14$}\node(abababab)(80,-10){$15$}
\node(aababaaab)(90,19){$15$}\node(ababaaaba)(90,5){$17$}\node(abababaaa)(90,0){$17$}\node[Nmr=0](ababababa)(90,-10){$1$}
\node[Nmr=0](aababaaaba)(100,19){$1$}
\node(ababaaabab)(100,5){$15$}\node(abababaaab)(100,0){$15$}
\node[Nmr=0](ababaaababa)(110,5){$1$}\node[Nmr=0](abababaaaba)(110,0){$1$}

\drawedge(1,a){$a$}\drawedge(1,b){$b$}
\drawedge(a,aa){$a$}\drawedge(a,ab){$b$}
\drawedge(aa,aaa){$a$}\drawedge(aa,aab){$b$}\drawedge(ab,aba){$a$}
\drawedge(aab,aaba){$a$}\drawedge(aba,abaa){$a$}\drawedge(aba,abab){$b$}
\drawedge(aaba,aabab){$b$}\drawedge(abab,ababa){$a$}
\drawedge(aabab,aababa){$a$}\drawedge(ababa,ababaa){$a$}\drawedge(ababa,ababab){$b$}
\drawedge(aababa,aababaa){$a$}\drawedge(aababa,aababab){$b$}
\drawedge(ababaa,ababaaa){$a$}\drawedge(ababab,abababa){$a$}
\drawedge(aababaa,aababaaa){$a$}\drawedge(aababab,aabababa){$a$}
\drawedge(ababaaa,ababaaab){$b$}\drawedge(abababa,abababaa){$a$}\drawedge(abababa,abababab){$b$}
\drawedge(aababaaa,aababaaab){$b$}
\drawedge(ababaaab,ababaaaba){$a$}\drawedge(abababaa,abababaaa){$a$}\drawedge(abababab,ababababa){$a$}
\drawedge(aababaaab,aababaaaba){$a$}
\drawedge(ababaaaba,ababaaabab){$b$}\drawedge(abababaaa,abababaaab){$b$}
\drawedge(ababaaabab,ababaaababa){$a$}\drawedge(abababaaab,abababaaaba){$a$}
\end{picture}
\caption{The bifix code $X$.}\label{figureBifixDegree5}
\end{figure}
In agreement with Theorem~\ref{theoremGroupCode}, the $F$-degree of $X$ is $5$. 
The $F$-minimal $\DD$-class is represented in Figure~\ref{figureFMinimalDClass}.
\begin{figure}[hbt]
\begin{picture}(120,80)(0,-40)
${\footnotesize
\def\rb{\hspace{2pt}\raisebox{0.8ex}{*}}\def\vh{\vphantom{\biggl(}}
    \begin{array}%
    {r|@{}l@{}c|@{}l@{}c|@{}l@{}c|@{}l@{}c|@{}l@{}c|@{}l@{}c|}%
    \multicolumn{1}{r}{}&\multicolumn{2}{c}{1,2,3,16,17}&\multicolumn{2}{c}{1,4,5,14,15}&\multicolumn{2}{c}{1,2,6,7,17}&
\multicolumn{2}{c}{1,4,8,9,15}&\multicolumn{2}{c}{1,2,6,10,11}&\multicolumn{2}{c}{1,2,3,12,14}\\
    \cline{2-13}
    1/2,4/3,6,15/& \vh &a^3 &\vh  &a^3b &\vh&a^3ba&\vh\rb&a^3bab&&&&\\
    8/9             &     &    &     &     &   &     &      &     &&&&\\
    \cline{2-13}
    1/2/11,17/6&\vh&ba^3 &\vh&&\vh\rb&&&&&&&\\
    7,9,10     &   &     &   &&      &&&&&&&\\
    \cline{2-13}
    1/3,6,15/9,14&\vh&aba^3&\vh\rb&&\vh&&&&&&&\\
    5,8/4        &   &     &      &&   &&&&&&&\\
    \cline{2-13}
    1/11,17/7,16&\vh\rb&baba^3&&&&&&bab&&&&\\
    3,6/2       &      &      &&&&&&   &&&&\\ 
    \cline{2-13}
    1/2,4/9,14&&&\vh\rb&&&&&&&&\vh\rb&\\
    3,6,15/5,12              &&&      &&&&&&&&      &\\
    \cline{2-13}
    1/2,4/3,6,15&&&&&&&&&\vh\rb&&&\\
    10/11       &&&&&&&&&      &&&\\
    \cline{2-13}
    \end{array}
}$
\end{picture}
\caption{The $F$-minimal $\DD$-class.}\label{figureFMinimalDClass}
\end{figure}
The word $a^3$ has rank $5$ and $\RR_F(a^3)=\{babaaa,babababaaa\}$.
The corresponding permutations defined on the image $\{1,2,3,16,17\}$ of $a^3$
are respectively
\begin{displaymath}
(1,2,16,3,17)\quad (1,17,16,2,3)
\end{displaymath}
which generate $A_5$.
\end{example}

The next example is from~\cite{Perrin2015}.
\begin{example}\label{exampleMorseTrivialGroup}
Let $F$ be the Thue-Morse set and let $\A$ be the automaton
represented in Figure~\ref{figMorse} on the left.
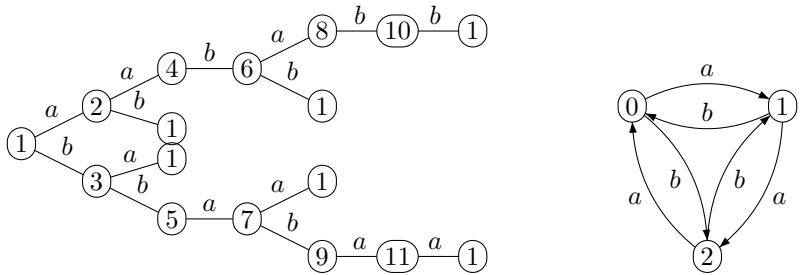
\begin{figure}[hbt]
\centering
\gasset{AHnb=0,Nadjust=wh}
\begin{picture}(120,35)
\put(0,0){
\begin{picture}(60,35)
\node(1)(0,15){$1$}\node(a)(10,20){$2$}\node(b)(10,10){$3$}
\node(aa)(20,25){$4$}\node(ab)(20,17){$1$}
\node(ba)(20,13){$1$}\node(bb)(20,5){$5$}
\node(aab)(30,25){$6$}\node(aaba)(40,30){$8$}\node(aabb)(40,20){$1$}
\node(aabab)(50,30){$10$}\node(aababb)(60,30){$1$}
\node(bba)(30,5){$7$}\node(bbaa)(40,10){$1$}\node(bbab)(40,0){$9$}
\node(bbaba)(50,0){$11$}\node(bbabaa)(60,0){$1$}

\drawedge(1,a){$a$}\drawedge(1,b){$b$}
\drawedge(a,aa){$a$}\drawedge(a,ab){$b$}
\drawedge(b,ba){$a$}\drawedge(b,bb){$b$}
\drawedge(aa,aab){$b$}\drawedge(aab,aaba){$a$}\drawedge(aab,aabb){$b$}
\drawedge(aaba,aabab){$b$}\drawedge(aabab,aababb){$b$}
\drawedge(bb,bba){$a$}\drawedge(bba,bbaa){$a$}\drawedge(bba,bbab){$b$}
\drawedge(bbab,bbaba){$a$}\drawedge(bbaba,bbabaa){$a$}
\end{picture}
}
\put(80,0){
\gasset{AHnb=1}
\begin{picture}(30,25)
\node(0)(0,20){$0$}\node(1)(20,20){$1$}\node(2)(10,0){$2$}

\drawedge[curvedepth=3](0,1){$a$}\drawedge[curvedepth=3](1,2){$a$}
\drawedge[curvedepth=3](2,0){$a$}
\drawedge[curvedepth=3,ELside=r](0,2){$b$}\drawedge[curvedepth=3,ELside=r](2,1){$b$}
\drawedge[curvedepth=3,ELside=r](1,0){$b$}
\end{picture}
}
\end{picture}
\caption{An automaton of $F$-degree $3$ with trivial
  $F$-group}
\label{figMorse}
\end{figure}
The word $aa$ has rank $3$ and image $I=\{1,2,4\}$.
The action on the images accessible from $I$
is given in Figure~\ref{figActionMorse}.
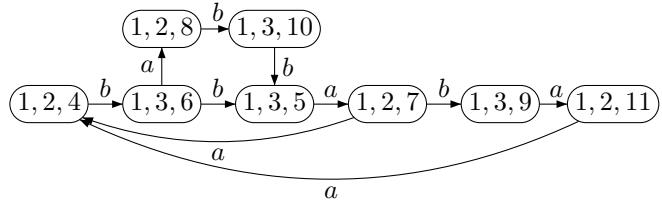
\begin{figure}[hbt]
\centering
\gasset{Nadjust=wh}
\begin{picture}(100,25)
\node(124)(0,10){$1,2,4$}\node(136)(15,10){$1,3,6$}
\node(135)(30,10){$1,3,5$}\node(127)(45,10){$1,2,7$}
\node(139)(60,10){$1,3,9$}\node(1211)(75,10){$1,2,11$}
\node(128)(15,20){$1,2,8$}\node(1310)(30,20){$1,3,10$}

\drawedge(124,136){$b$}\drawedge(136,135){$b$}\drawedge(135,127){$a$}
\drawedge(127,139){$b$}\drawedge(139,1211){$a$}
\drawedge[curvedepth=10](1211,124){$a$}
\drawedge(136,128){$a$}\drawedge(128,1310){$b$}\drawedge(1310,135){$b$}
\drawedge[curvedepth=5](127,124){$a$}
\end{picture}
\caption{The action on the minimal images}\label{figActionMorse}
\end{figure}
All words with image $\{1,2,4\}$ end with $aa$.  The paths
returning for the first time to $\{1,2,4\}$ are labeled by the set
$\RR_F(aa)=\{b^2a^2,bab^2aba^2,bab^2a^2,b^2aba^2\}$. Thus $\rank_\A(F)=3$ by~\cite[Theorem 3.1]{Perrin2015}.
Moreover each of the
words of $\RR_F(a^2)$ defines the trivial permutation on the set $\{1,2,4\}$.
Thus  $G_\A(F)$ is trivial. 

The fact that $d_\A(F)=3$ and that $G_\A(F)$ is trivial can be seen directly
as follows. Consider the group automaton $\B$ represented in Figure~\ref{figMorse} on the right and corresponding to the map sending each word to the
difference
modulo $3$ of the number of occurrences of $a$ and $b$. There is a reduction $\rho$ from $\A$ onto $\B$ such that $1\mapsto 0$,
$2\mapsto 1$, and $4\mapsto 2$. This accounts for the fact that
$d_\A(S)=3$. Moreover, one may verify that
 any return word $x$ to $a^2$ has equal number of $a$ and $b$ (if $x=uaa$ then $aauaa$
is in $F$, which implies that $aua$ and thus $uaa$ have the same number of $a$
and $b$). This implies that the permutation
$\varphi_\B(x)$ is the identity, and therefore also the restriction of
$\varphi_\A(x)$ to $I$. 
\end{example}
\begin{example}\label{exampleThueMorse8}
Consider again the Thue-Morse substitution $\tau$
and the Thue-Morse set $F$ as in Example~\ref{exampleThueMorse4}.
Let $h$ be the morphism
$h:a\mapsto (123),b\mapsto (345)$ from $A^*$ onto the alternating group $A_5$
(already used in Example~\ref{exampleMorphismAC3}). One may verify that $\tau$
has $h$-order $6$ and thus, by Corollary~\ref{corollaryPresentation}, $h$ extends to a surjective continuous morphim from
any maximal subgroup of $J(\varphi)$ onto $A_5$. 

 Let $Z$ be the group code
generating the submonoid stabilizing $1$ and let $X=Z\cap F$.
The $F$-maximal bifix code $X$ is represented in Figure~\ref{figureBifixDegree5Morse}.
We represent in Figure~\ref{figureBifixDegree5Morse} only the nodes corresponding
to right special words, that is, vertices with two sons.
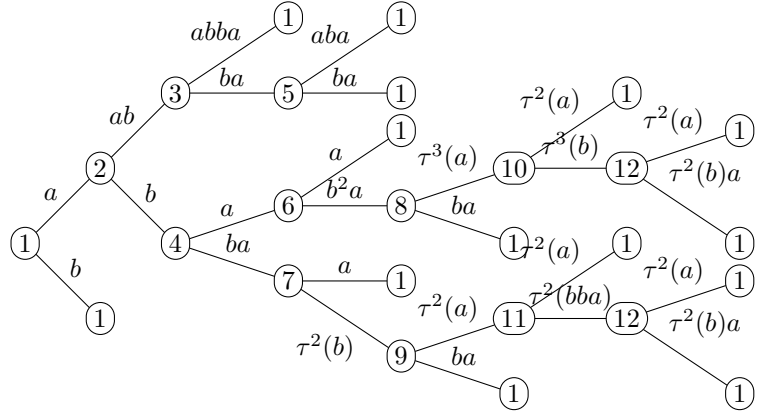
\begin{figure}[hbt]
\centering\gasset{Nadjust=wh,AHnb=0}
\begin{picture}(80,50)(0,-10)
\node(1)(0,10){$1$}
\node(2)(10,20){$2$}\node(b)(10,0){$1$}
\node(3)(20,30){$3$}\node(4)(20,10){$4$}
\node(x1)(35,40){$1$}\node(5)(35,30){$5$}
\node(6)(35,15){$6$}\node(7)(35,5){$7$}
\node(x2)(50,40){$1$}\node(x3)(50,30){$1$}
\node(x4)(50,25){$1$}\node(8)(50,15){$8$}
\node(x9)(50,5){$1$}\node(9)(50,-5){$9$}
\node(10)(65,20){$10$}\node(x8)(65,10){$1$}
\node(11)(65,0){$11$}\node(x13)(65,-10){$1$}
\node(x5)(80,30){$1$}\node(12)(80,20){$12$}
\node(x10)(80,10){$1$}\node(12b)(80,0){$12$}
\node(x6)(95,25){$1$}\node(x7)(95,10){$1$}
\node(x11)(95,5){$1$}\node(x12)(95,-10){$1$}

\drawedge(1,2){$a$}\drawedge(1,b){$b$}
\drawedge(2,3){$ab$}\drawedge(2,4){$b$}
\drawedge(3,x1){$abba$}\drawedge(3,5){$ba$}
\drawedge(4,6){$a$}\drawedge(4,7){$ba$}
\drawedge(5,x2){$aba$}\drawedge(5,x3){$ba$}
\drawedge(6,x4){$a$}\drawedge(6,8){$b^2a$}
\drawedge(7,x9){$a$}\drawedge[ELside=r](7,9){$\tau^2(b)$}
\drawedge(8,10){$\tau^3(a)$}\drawedge(8,x8){$ba$}
\drawedge(9,11){$\tau^2(a)$}\drawedge(9,x13){$ba$}
\drawedge(10,x5){$\tau^2(a)$}\drawedge(10,12){$\tau^3(b)$}
\drawedge(11,x10){$\tau^2(a)$}\drawedge(11,12b){$\tau^2(bba)$}
\drawedge(12,x6){$\tau^2(a)$}\drawedge(12,x7){$\tau^2(b)a$}
\drawedge(12b,x11){$\tau^2(a)$}\drawedge(12b,x12){$\tau^2(b)a$}
\end{picture}
\caption{The bifix code $X$.}\label{figureBifixDegree5Morse}
\end{figure}

The image of $\tau^4(b)$ is $\{1,3,4,9,10\}$ and thus it is minimal.
The action on its image is shown in Figure~\ref{figureMinimalImages}.
The return words to $\tau^4(b)$ are $\tau^4(b),\tau^3(a)$ and
$\tau^5(ab)$. The permutations on the image of $\tau^4(b)$ are the $3$
cycles of length $5$ indicated in Figure~\ref{figureMinimalImages}.
Since they generate the group $A_5$, we have $G_X(F)=A_5$.

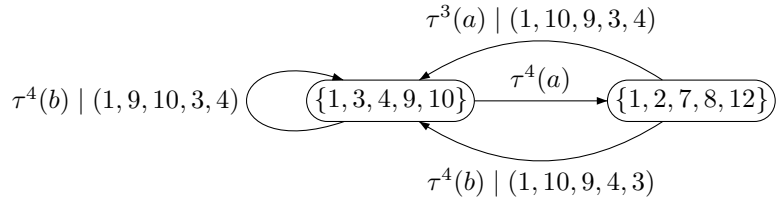
\begin{figure}[hbt]
\centering\gasset{Nadjust=wh}
\begin{picture}(50,20)(-20,0)
\node(1)(0,10){$\{1,3,4,9,10\}$}\node(2)(40,10){$\{1,2,7,8,12\}$}

\drawloop[loopangle=180](1){$\tau^4(b)\mid (1,9,10,3,4)$}
\drawedge(1,2){$\tau^4(a)$}
\drawedge[curvedepth=-8,ELside=r](2,1){$\tau^3(a)\mid (1,10,9,3,4)$}
\drawedge[curvedepth=8](2,1){$\tau^4(b)\mid (1,10,9,4,3)$}
\end{picture}
\caption{The action on the minimal images.}\label{figureMinimalImages}
\end{figure}

\end{example}
\bibliographystyle{plain}
\bibliography{profinite}

\def\cprime{$'$}
\begin{thebibliography}{10}

\bibitem{Almeida2002}
Jorge Almeida.
\newblock Dynamics of implicit operations and tameness of pseudovarieties of
  groups.
\newblock {\em Trans. Amer. Math. Soc.}, 354(1):387--411, 2002.

\bibitem{Almeida2005}
Jorge Almeida.
\newblock Profinite semigroups and applications.
\newblock In {\em Structural theory of automata, semigroups, and universal
  algebra}, volume 207 of {\em NATO Sci. Ser. II Math. Phys. Chem.}, pages
  1--45. Springer, Dordrecht, 2005.
\newblock Notes taken by Alfredo Costa.

\bibitem{AlmeidaCosta2009}
Jorge Almeida and Alfredo Costa.
\newblock Infinite-vertex free profinite semigroupoids and symbolic dynamics.
\newblock {\em J. Pure Appl. Algebra}, 213(5):605--631, 2009.

\bibitem{AlmeidaCosta2013}
Jorge Almeida and Alfredo Costa.
\newblock Presentations of {S}ch{\"u}tzenberger groups of minimal subshifts.
\newblock {\em Israel J. Math.}, 196(1):1--31, 2013.

\bibitem{AlmeidaCosta2015}
Jorge Almeida and Alfredo Costa.
\newblock A geometric interpretation of the {S}ch\"utzenberger group of a
  minimal subshift.
\newblock {\em Ark. Mat.}, 54(2):243--275, 2016.

\bibitem{AlmeidaSteinberg2009}
Jorge Almeida and Benjamin Steinberg.
\newblock Rational codes and free profinite monoids.
\newblock {\em J. Lond. Math. Soc. (2)}, 79(2):465--477, 2009.

\bibitem{Almeida2005b}
Zh. Alme{\u\i}da.
\newblock Profinite groups associated with weakly primitive substitutions.
\newblock {\em Fundam. Prikl. Mat.}, 11(3):13--48, 2005.

\bibitem{BalkovaPelantovaSteiner2008}
{\soft{L}}ubom{\'{\i}}ra Balkov{\'a}, Edita Pelantov{\'a}, and Wolfgang
  Steiner.
\newblock Sequences with constant number of return words.
\newblock {\em Monatsh. Math.}, 155(3-4):251--263, 2008.

\bibitem{BerstelDeFelicePerrinReutenauerRindone2012}
Jean Berstel, Clelia De~Felice, Dominique Perrin, Christophe Reutenauer, and
  Giuseppina Rindone.
\newblock Bifix codes and {S}turmian words.
\newblock {\em J. Algebra}, 369:146--202, 2012.

\bibitem{BerstelPerrinReutenauer2009}
Jean Berstel, Dominique Perrin, and Christophe Reutenauer.
\newblock {\em Codes and Automata}.
\newblock Cambridge University Press, 2009.

\bibitem{BertheDeFeliceDolceLeroyPerrinReutenauerRindone2013a}
Val\'erie Berth\'e, Clelia De~Felice, Francesco Dolce, Julien Leroy, Dominique
  Perrin, Christophe Reutenauer, and Giuseppina Rindone.
\newblock Acyclic, connected and tree sets.
\newblock {\em Monatsh. Math.}, 176:521--550, 2015.

\bibitem{BertheDeFeliceDolceLeroyPerrinReutenauerRindone2015}
Val{\'e}rie Berth{\'e}, Clelia De~Felice, Francesco Dolce, Julien Leroy,
  Dominique Perrin, Christophe Reutenauer, and Giuseppina Rindone.
\newblock The finite index basis property.
\newblock {\em J. Pure Appl. Algebra}, 219(7):2521--2537, 2015.

\bibitem{BertheDeFeliceDolceLeroyPerrinReutenauerRindone2013m}
Val\'erie Berth\'e, Clelia De~Felice, Francesco Dolce, Julien Leroy, Dominique
  Perrin, Christophe Reutenauer, and Giuseppina Rindone.
\newblock Maximal bifix decoding.
\newblock {\em Dicrete Math.}, 338:725--742, 2015.

\bibitem{CoulboisSapirWeil2003}
Thierry Coulbois, Mark Sapir, and Pascal Weil.
\newblock A note on the continuous extensions of injective morphisms between
  free groups to relatively free profinite groups.
\newblock {\em Publ. Mat.}, 47(2):477--487, 2003.

\bibitem{DroubayJustinPirillo2001}
Xavier Droubay, Jacques Justin, and Giuseppe Pirillo.
\newblock Episturmian words and some constructions of de {L}uca and {R}auzy.
\newblock {\em Theoret. Comput. Sci.}, 255(1-2):539--553, 2001.

\bibitem{DurandLeroyRichomme2013}
Fabien Durand, Julien Leroy, and Gwena{\"e}l Richomme.
\newblock Do the properties of an {$S$}-adic representation determine factor
  complexity?
\newblock {\em J. Integer Seq.}, 16(2):Article 13.2.6, 30, 2013.

\bibitem{PytheasFogg2002}
N.~Pytheas Fogg.
\newblock {\em Substitutions in dynamics, arithmetics and combinatorics},
  volume 1794 of {\em Lecture Notes in Mathematics}.
\newblock Springer-Verlag, Berlin, 2002.
\newblock Edited by V. Berth{\'e}, S. Ferenczi, C. Mauduit and A. Siegel.

\bibitem{GlenJustin2009}
Amy Glen and Jacques Justin.
\newblock Episturmian words: a survey.
\newblock {\em Theor. Inform. Appl.}, 43:403--442, 2009.

\bibitem{Hall1949}
Marshall Hall, Jr.
\newblock Coset representations in free groups.
\newblock {\em Trans. Amer. Math. Soc.}, 67:421--432, 1949.

\bibitem{Hall1950}
Marshall Hall, Jr.
\newblock A topology for free groups and related groups.
\newblock {\em Ann. of Math. (2)}, 52:127--139, 1950.

\bibitem{JustinVuillon2000}
Jacques Justin and Laurent Vuillon.
\newblock Return words in {S}turmian and episturmian words.
\newblock {\em Theor. Inform. Appl.}, 34(5):343--356, 2000.

\bibitem{KloudaStarosta2014}
Karel Klouda and Stepan Starosta.
\newblock Characterization of circular dol systems.
\newblock 2014.
\newblock \url{http:\\arxiv.org/abs/1401.0038}.

\bibitem{Knuth1998}
Donald~E. Knuth.
\newblock {\em The art of computer programming. {V}ol. 2}.
\newblock Addison-Wesley, Reading, MA, 1998.
\newblock Seminumerical algorithms, Third edition [of MR0286318].

\bibitem{Koblitz1984}
Neal Koblitz.
\newblock {\em {$p$}-adic numbers, {$p$}-adic analysis, and zeta-functions},
  volume~58 of {\em Graduate Texts in Mathematics}.
\newblock Springer-Verlag, New York, second edition, 1984.

\bibitem{Lallement1979}
G{\'e}rard Lallement.
\newblock {\em Semigroups and combinatorial applications}.
\newblock John Wiley \& Sons, New York-Chichester-Brisbane, 1979.
\newblock Pure and Applied Mathematics, A Wiley-Interscience Publication.

\bibitem{Lenstra2005}
Hendrick Lenstra.
\newblock Profinite {F}ibonacci numbers.
\newblock {\em Nieuw Archief voor Wiskunde}, 6:297--300, 2005.

\bibitem{LindMarcus1995}
Douglas Lind and Brian Marcus.
\newblock {\em An introduction to symbolic dynamics and coding}.
\newblock Cambridge University Press, Cambridge, 1995.

\bibitem{LyndonSchupp2001}
Roger~C. Lyndon and Paul~E. Schupp.
\newblock {\em Combinatorial group theory}.
\newblock Classics in Mathematics. Springer-Verlag, Berlin, 2001.
\newblock Reprint of the 1977 edition.

\bibitem{MargolisSapirWeil1998}
S.~Margolis, M.~Sapir, and P.~Weil.
\newblock Irreducibility of certain pseudovarieties.
\newblock {\em Comm. Algebra}, 26(3):779--792, 1998.

\bibitem{Mosse1992}
Brigitte Moss{\'e}.
\newblock Puissances de mots et reconnaissabilit\'e des points fixes d'une
  substitution.
\newblock {\em Theoret. Comput. Sci.}, 99(2):327--334, 1992.

\bibitem{Perrin2015}
Dominique Perrin.
\newblock Codes and automata in minimal sets.
\newblock In {\em Combinatorics on Words - 10th International Conference,
  {WORDS} 2015, Kiel, Germany, September 14-17, 2015, Proceedings}, pages
  35--46, 2015.

\bibitem{Reutenauer1979}
Christophe Reutenauer.
\newblock Une topologie du mono\"\i de libre.
\newblock {\em Semigroup Forum}, 18(1):33--49, 1979.

\bibitem{RibesZalesskii2010}
Luis Ribes and Pavel Zalesskii.
\newblock {\em Profinite groups}, volume~40 of {\em Ergebnisse der Mathematik
  und ihrer Grenzgebiete. 3. Folge. A Series of Modern Surveys in Mathematics
  [Results in Mathematics and Related Areas. 3rd Series. A Series of Modern
  Surveys in Mathematics]}.
\newblock Springer-Verlag, Berlin, second edition, 2010.

\bibitem{Willard2004}
Stephen Willard.
\newblock {\em General topology}.
\newblock Dover Publications, Inc., Mineola, NY, 2004.
\newblock Reprint of the 1970 original [Addison-Wesley, Reading, MA;
  MR0264581].

\end{thebibliography}
\printindex

\end{document}